\documentclass[reqno]{amsart}

\usepackage{verbatim}

\usepackage{hyperref}
\usepackage{mathrsfs}
\usepackage{amssymb}
\usepackage{latexsym}

\usepackage{comment}

\usepackage[mathscr]{eucal}
\usepackage{slashed}

\usepackage[all]{xy}
\usepackage{pb-diagram,pb-xy}
\usepackage{setspace}

\usepackage[a4paper, hmargin=3.0cm, tmargin=3.2cm, bmargin=3.2cm]{geometry}
\newcommand{\impl}[2]{{\bf #1 $\mathbf{\Rightarrow}$ #2 :}}

\numberwithin{equation}{section}

\newtheorem{MainThm}{Pretheorem}

\theoremstyle{definition}

\newtheorem{defn}[equation]{Definition}

\newtheorem{definition-proposition}[equation]{Definition/Proposition}

\newtheorem{remark}[equation]{Remark}
\newtheorem{example}[equation]{Example}
\newtheorem{impexample}[equation]{Important Example}

\theoremstyle{plain}
\newtheorem{lem}[equation]{Lemma}
\newtheorem{thm}[equation]{Theorem}

\newtheorem{prop}[equation]{Proposition}
\newtheorem{proposition}[equation]{Proposition}

\newtheorem{corollary}[equation]{Corollary}
\newcommand{\umbruch}{\mbox{}}

\newcommand{\gra}{\operatorname{gra}}

\newcommand{\Tr}{\operatorname{Tr}}

\newcommand{\bC}{\mathbb{C}}
\newcommand{\bD}{\mathbb{D}}

\newcommand{\bN}{\mathbb{N}}

\newcommand{\bR}{\mathbb{R}}
\newcommand{\bS}{\mathbb{S}}

\newcommand{\bZ}{\mathbb{Z}}

\newcommand{\bK}{\mathbb{K}}

\newcommand{\gJ}{\mathbf{J}}

\newcommand{\gA}{\mathbf{A}}
\newcommand{\gB}{\mathbf{B}}

\newcommand{\ev}{\mathrm{ev}}

\newcommand{\cL}{\mathcal{L}}

\newcommand{\cU}{\mathcal{U}}

\newcommand{\Aut}{\operatorname{Aut}}

\newcommand{\bott}{\operatorname{bott}}

\newcommand{\thom}{\operatorname{thom}}
\newcommand{\Mor}{\operatorname{mor}}

\newcommand{\Lin}{\mathbf{Lin}}

\newcommand{\Cl}{\mathbf{Cl}}

\newcommand{\dom}{\mathrm{dom}}
\newcommand{\eps}{\epsilon}

\newcommand{\End}{\operatorname{End}}

\newcommand{\im}{\operatorname{Im}}
\newcommand{\ind}{\operatorname{index}}

\newcommand{\Kom}{\mathbf{Kom}}

\newcommand{\normalize}[1]{\frac{#1}{(1+{#1}^{2})^{1/2}}}

\newcommand{\pr}{\operatorname{pr}}

\newcommand{\gI}{\mathbf{I}}

\newcommand{\gC}{\mathbf{C}}

\newcommand{\red}{\mathrm{red}}

\newcommand{\spec}{\operatorname{spec}}

\newcommand{\Spin}{\mathrm{Spin}}
\newcommand{\supp}{\operatorname{supp}}

\newcommand{\twomatrix}[4]{\begin{pmatrix} #1 & #2 \\ #3 & #4  \end{pmatrix}}

\newcommand{\Tot}{\mathrm{Tot}}

\newcommand{\scpr}[1]{\langle #1 \rangle}

\newcommand{\cstar}{\mathrm{C}^{\ast}}

\newcommand{\Dir}{\slashed{\mathfrak{D}}}
\newcommand{\spinor}{\slashed{\mathfrak{S}}}

\newcommand{\coker}{\operatorname{coker}}

\newcommand{\smb}{\mathrm{smb}}

\newcommand{\norm}[1]{\| #1 \|}

\newcommand{\vol}{\mathrm{vol}}

\title[Elliptic regularity for Dirac operators on families of noncompact manifolds]{Elliptic regularity for Dirac operators on families of noncompact manifolds}

\author{Johannes Ebert}
\thanks{Partially supported by the SFB 878 ``Groups, Geometry and Actions''}
\address{Mathematisches Institut, Universit\"at M\"unster\\
Einsteinstra{\ss}e 62\\
48149 M\"unster\\
Bundesrepublik Deutschland}

\email{johannes.ebert@uni-muenster.de}
\date{\today}


\begin{document}

\begin{abstract}
We develop elliptic regularity theory for Dirac operators in a very general framework: we consider Dirac operators linear over $\cstar$-algebras, on noncompact manifolds, and in families which are not necessarily locally trivial fibre bundles. 
\end{abstract}

\maketitle

\tableofcontents

\clearpage


The simplest situation where an elliptic operator has an index is when $M$ is a closed Riemannian manifold, $V_0, V_1 \to M$ are two smooth Hermitian vector bundles and $D_0: \Gamma (M;V_0)\to \Gamma (M;V_1)$ is an elliptic differential operator of order $1$ between the spaces of smooth sections of these vector bundles. Under these hypotheses, $D_0$ induces a Fredholm operator $W^1 (M;V_0)\to L^2 (M;V_1)$ from the Sobolev space of order $1$ to the space of $L^2$-sections. The index of $D_0$ is $\ind (D_0):=\dim (\ker (D_0))-\dim (\coker (D_0))\in \bZ =K^0 (*)$. 
For many purposes, a slightly different approach to the index is more convenient. The operator $D_0$ has a formal adjoint $D_0^*: \Gamma (M;V_1)\to \Gamma (M;V_0)$. The formulas
\[
\iota:= \twomatrix{1}{}{}{-1}; \;  D := \twomatrix{}{D_0^*}{D_0}{}
\]
define a \emph{grading} on the vector bundle $V=V_0\oplus V_1$ and a formally self-adjoint elliptic operator $D: \Gamma(M;V) \to \Gamma(M;V)$. It can be shown that $D: \Gamma (M;V) \to L^2 (M;V)$ is an essentially self-adjoint unbounded operator on the Hilbert space $L^2 (M;V)$, so that one can construct the bounded transform $\normalize{D}$ via functional calculus. Then $\normalize{D}$ is a bounded self-adjoint Fredholm operator on $L^2(M;V)$. The index of $D_0$ is encoded in the grading $\iota$: since $D$ is odd, i.e. $D\iota+\iota D=0$, the finite-dimensional vector space $\ker (\normalize{D})=\ker (D)$ is invariant under the involution $\iota$, and $\ind (D)=\Tr (\iota|_{\ker (\normalize{D})})$. 
Many generalizations of these results are well-known and can be found in the literature.

\begin{enumerate}
\item {\bf Noncompact manifolds.} If $M$ is not compact and $D$ is a formally self-adjoint differential operator of order $1$ on a vector bundle $V\to M$, then $D: \Gamma_c (M;V) \to L^2 (M;V)$ is not necessarily essentially self-adjoint. However, if $(M,D)$ is \emph{complete} (see Definition \ref{defn:complete-operator-on-mfd}), then $D$ is essentially self-adjoint. This is a classical result by Wolf \cite{Wolf} and Chernoff \cite{Chernoff}, see also \cite[Proposition 10.2.10]{HR}. 
Even if $D$ is essentially self-adjoint, it is not necessarily the case that $\normalize{D}$ is Fredholm. This happens for example if $D$ is invertible at infinity, i.e. there is a compact $K \subset M$ and $c>0$ such that $\norm{D u} \geq c \norm{u}$ whenever the support of $u$ does not meet $K$. 
\item {\bf Families of operators.} Instead of a single manifold $M$, we consider a fibre bundle $\pi:M \to X$ over a compact base space $X$ with closed fibres. For families $D= (D_x)_{x \in X}$ of elliptic differential operators on the fibres $M_x$, Atiyah and Singer \cite{ASIoeoIV} defined a family index $\ind (D) \in K^0 (X)$. This can be thought of as the formal difference of the vector bundles $\ker (D)-\coker (D)$, but there is the problem that the kernels and cokernel of $D_x$ only form a vector bundle over $X$ under the very special hypothesis that their dimensions are constant. It is conceptually clearer to generalize the definition of $K^0 (X)$, so that the family of Fredholm operators $(\normalize{D_x})_{x \in X}$ on the Hilbert spaces $(L^2 (M_x, V_x)_x)$ directly represents an element of $K^0 (X)$. 
\item {\bf Operators linear over $\cstar$-algebras.} In the context of the Novikov conjecture, Mishchen\-ko and Fomenko \cite{MF} developed index theory for differential operators $D$ acting on the sections of a bundle whose fibres are finitely generated projective modules over a group $\cstar$-algebra $\gC^* (G)$ (here $G$ is a discrete group) and whose base is a closed manifold. In this case, the index of $D$ is an element of the $K$-theory group $K_0 (\gC^* (G))$.
The same construction was used by Rosenberg \cite{Ros} in connection with metrics of positive scalar curvature. More generally, $\gC^* (G)$ can be replaced by an arbitrary $\cstar$-algebra $\gA$. The basic regularity results for closed manifolds are due to \cite{MF} and have been generalized to noncompact manifolds in \cite{Stolz2}, \cite{Bunke} and \cite{HPS}.
\end{enumerate}
In \cite{JEIndex2}, we will consider the following very general situation. Instead of a single manifold, we consider a submersion $\pi:M \to X$ with possibly noncompact fibres $M_x:= \pi^{-1}(x)$, and which is possibly not a locally trivial fibre bundle. We consider bundles of finitely generated projective $\gA$-modules $V\to M$ and families of $\gA$-linear elliptic differential operators $(D_x)_x$ on $V|_{M_x}$. We want to formulate precise conditions under which $D$ has an index in the $K$-theory group $K^0 (X;\gA)$. Here $K^0 (\_; \gA)$ is the generalized cohomology theory represented by the $K$-theory spectrum of the $\cstar$-algebra $\gA$. Finally, we are also interested in graded and Real $\cstar$-algebras and operators obeying additional Clifford symmetries, which then have indices in the groups $KR^* (X; \gA)$, for varying values of $*$. By this level of generality, all the above situations are covered.
We do not attempt to study more general operators, such as differential operators of order $\geq 2$ or pseudo-differential operators. 

Let us give a brief overview of the paper, stating the results only informally. After setting up the relevant definitions, Section \ref{sec:diffops-overcstar} is devoted to the proof of the following result, see Theorem \ref{chernoff-theorem} for the precise statement.
\begin{MainThm}\label{self-adjointness-intro}
Let $D$ be an  $\gA$-linear differential operator of order $1$ on a manifold $M$. Under a suitable completeness hypothesis, $D$ is self-adjoint in the sense of the theory of unbounded operators on Hilbert-$\gA$-modules.
\end{MainThm}
This has been proven before by Hanke, Pape and Schick \cite{HPS} by a different method; our proof uses a ``local-to-global principle'' for self-adjoint operators due to Kaad and Lesch \cite{KL}. 
In section \ref{sec:parametrized-regularity}, we study families of operators on noncompact manifolds (the families of manifolds are submersions, not fibre bundles). It turns out that the theory of continuous fields of Banach spaces introduced by Dixmier and Douady \cite{DD} is tailor-made for this problem. We define unbounded self-adjoint operator families on continuous fields of Hilbert $\gA$-modules, and discuss the functional calculus for those. Once this machinery is in place, Theorem \ref{self-adjointness-intro} can be generalized with very little effort.
\begin{MainThm}
Let $\pi: M \to X$ be a submersion and let $D$ be a family of $\gA$-linear differential operators of order $1$ on the fibres of $\pi$. Under a suitable completeness hypothesis, $D$ defines a self-adjoint unbounded operator family on a continuous field of Hilbert-$\gA$-modules over $X$, son that one can define the bounded transform $\normalize{D}$.
\end{MainThm}
In subsection \ref{subsec:fredholmtheory-dirac}, we study conditions which imply that $\normalize{D}$ is a Fredholm family (invertible modulo compact operator families). The most important result in this direction is Theorem \ref{fredholmness:invertibility-at-infty}, which may be stated informally as follows.
\begin{MainThm}\label{fredholmness-introdu}
Assume that $D$ is invertible at infinity. Then $\normalize{D}$ is a Fredholm family.
\end{MainThm}
In the unparametrized case, this is well-known, see e.g. \cite{HR} or \cite{Bunke}.
In the final section \ref{sec:KTheory}, we show how to define the topological $K$-theory $K^{*}(X,Y;\gA)$ in terms of Fredholm families on continuous fields of Hilbert-$\gA$-modules. In the situation of Theorem \ref{fredholmness-introdu}, this allows us to define the index $\ind(D) \in K^0 (X;\gA)$.

\section{Differential operators linear over \texorpdfstring{$\cstar$}--algebras}\label{sec:diffops-overcstar}

\subsection{Definitions from abstract functional analysis}\label{subsec:defns-abstractFA}

\subsubsection{$\cstar$-algebras and Hilbert modules}

Throughout this paper, $\gA$ will denote a complex $\cstar$-algebra with unit. There are two types of extra structure on $\gA$ that we are interested in. The first is a $\bZ/2$-\emph{grading}, which is a norm-preserving $\bC$-linear $*$-automorphism $\alpha$ of $\gA$ such that $\alpha^2=1$.
For more information on graded $\cstar$-algebras, the reader should consult \cite[\S 14]{Bla}.
The second extra structure on $\gA$ is a \emph{Real structure}, which is a \emph{conjugate-linear} norm-preserving $*$-automorphism $\kappa$ of $\gA$ such that $\kappa^2=1$. Often, we write the Real structure as $\overline{a}:= \kappa (a)$. 
For more details on Real $\cstar$-algebras, we refer to \cite{Schr} and \cite{Goodearl}.
If $\gA$ has both, a grading and a Real structure, then we require that they are compatible in the sense that $\kappa \alpha = \alpha \kappa$. 

We are mainly interested in a few types of $\cstar$-algebras. If $H$ is a complex Hilbert space, we denote by $\Lin (H)$ and $\Kom (H)$ the $\cstar$-algebras of bounded and compact operators on $H$. A grading and Real structure on $H$ defines such structures on $\Lin (H)$ and $\Kom (H)$. 

If $X$ is a locally compact space, then $\gC_0 (X)$ is the $\cstar$-algebra of complex-valued functions on $X$ that vanish at infinity, equipped with the trivial grading ($\alpha=1$). An involution $\tau: X \to X$ (a ``Real structure'' in the sense of \cite{AtiKR}), defines a Real structure on $\gC_0 (X)$ by $(\kappa f)(x):= \overline{f(\tau (x))}$. More generally, if $\gA$ is any $\cstar$-algebra, then we let $\gA_0 (X)$ be the $\cstar$-algebra of all continuous functions $f:X \to \gA$ which vanish at infinity.

Next, we consider the \emph{Clifford algebras}. Let $V,W$ be (finite-dimensional) euclidean vector spaces. The Clifford algebra $\Cl (V \oplus W^-)$ is the $\gC$-algebra generated by the vector space $V \oplus W$, subject to the relations
\[
(v,w) (v',w') + (v',w') (v,w) = - 2\scpr{v,v'} + 2 \scpr{w,w'}.
\]
For $(v,w) \in V \oplus W$, we set $(v,w)^*:=   (-v,w)$ and extend this to a $\gC$-antilinear antiautomorphism of $\Cl (V \oplus W^-)$. A Real structure on $\Cl (V \oplus W^-)$ is defined on generators by $\overline{(v,w)}: = (v,w)$ (and then extended to a $\bC$-antilinear automorphism). 
The grading on $\Cl (V \oplus W^-)$ is given by $\alpha (v,w)= - (v,w)$ (and then extended to a $\bC$-linear automorphism). To define the norm on $\Cl (V \oplus W^-)$, recall that the exterior algebra $\Lambda^* (V \oplus W)$ has a natural inner product. For $x \in V\oplus W$, let $\eps_x:\Lambda^* (V \oplus W)\to \Lambda^* (V \oplus W) $ be the exterior product with $x$. Set
\[
c(v,w):= \eps_{(v,w)} + \eps_{(-v,w)}^* ;
\]
this extends to an injective $*$-algebra homomorphism $c:\Cl (V \oplus W^-)\to \Lin (\Lambda^* (V \oplus W))$. The norm on $\Cl(V\oplus W^-)$ is defined by $\norm{x}:= \norm{c(x)}$. We write $\Cl^{p,q}:= \Cl ( \bR^p  \oplus (\bR^q)^-)$. 

The last type of $\cstar$-algebras we want to consider are the group $\cstar$-algebras. Let $G$ be a countable discrete group. There are two $\cstar$-algebras associated with $G$, the \emph{reduced group $\cstar$-algebra} $\gC^*_{\red}(G)$ and the \emph{maximal group $\cstar$-algebra} $\gC^*_{\max}(G)$. Both are completions of the complex group ring $\bC [G]$, with respect to two norms. The maximal norm of $x \in \bC [ G]$ is $\norm{x}_{\max}:= \sup_{\rho} \norm{\rho(x)}$, where $\rho$ ranges over all unitary representations of $G$ on Hilbert spaces, while the reduced norm is $\norm{x}_{\red} := \norm{\rho_0 (x)}$, where $\rho_0: G \to \Lin (\ell^2 (G))$ is the regular representation. If there is no need to specify which group algebra is used, we just write $\gC^* (G)$. The conjugation on $\bC [G]$ induces a Real structure, and we only consider the trivial grading $\alpha=1$ on $\gC^* (G)$.

\begin{defn}\label{defn:hilbertmodule}
Let $\gA$ be a unital complex $\cstar$-algebra. 
A \emph{pre-Hilbert-$\gA$-module} is a right-$\gA$-module $E$, together with an $\gA$-valued inner product, i.e. a map $E \times E \to \gA$, $(x,y) \mapsto \scpr{x,y}$ such that 
\begin{enumerate}
\item $\scpr{x_0+x_1,y_0+y_1} = \scpr{x_0,y_0}+ \scpr{x_1,y_0}+ \scpr{x_0,y_1}+ \scpr{x_1,y_1}$,
\item $ \scpr{x,ya} = \scpr{x,y} a$,
\item $\scpr{x,y} = \scpr{y,x}^*$ and
\item $\scpr{x,x} \geq 0$ and $\scpr{x,x}=0 \Rightarrow x=0$
\end{enumerate}
hold for all $x,x_i,y,y_i \in E$ and $a \in \gA$. Note that $\scpr{x,x} \in \gA$, and the condition $\scpr{x,x} \geq 0$ is to be understood in the sense that it is a positive element of that $\cstar$-algebra.
\end{defn}

The scalar product induces a norm on $E$, namely $\norm{x} := \norm{\scpr{x,x}}^{1/2}$.
We say that $E$ is a \emph{Hilbert-$\gA$-module}, if the vector space $E$ is complete with respect to this norm. We refer the reader to \cite{Lance}, \cite[\S 15]{WO} for more background on Hilbert modules. If $\gA$ has a grading and a Real structure, we can talk about Real and graded Hilbert modules. It is required that $E$ has a Real structure and a grading, and that the structure maps $E \odot \gA \to E$ and $E \odot E \to \gA$ are compatible with the Real structure and grading on both sides (with the symbol $\odot$, we denote the algebraic tensor product of complex vector spaces).
In a Hilbert-$\gA$-module, the Cauchy-Schwarz inequality 
\begin{equation}\label{cauchy-schwarz}
\norm{\scpr{x,y}} \leq \norm{x} \norm{y}
\end{equation}
holds: by \cite[Proposition 1.1]{Lance}, one has $\scpr{x,y}\scpr{x,y}^*=\scpr{x,y} \scpr{y,x} \leq \norm{x}^2 \scpr{y,y}$. These elements are positive in $\gA$, and so by the $\cstar$-identity
\[
\norm{\scpr{x,y}}^2=\norm{\scpr{x,y}\scpr{x,y}^*} \leq \norm{x}^2 \norm{y}^2.
\]
The reader should recall at this point the basic notions of adjointable operators and compact operators of Hilbert $\gA$-modules \cite[\S 15.2]{WO}. By $\Lin_{\gA} (E)$, we denote the $\cstar$-algebra of adjointable operators $E$, and 
by $\Kom_{\gA}(E)\subset \Lin_{\gA}(E)$ the ideal of compact operators on $E$. 

\subsubsection{Bundles of projective $\gA$-modules}

Let $P$ be a finitely generated projective right $\gA$-module, equipped with a Banach norm such that the multiplication map $P \times \gA \to \gA$ is continuous. It is not hard to see that there is an $\gA$-valued inner product on $P$ which induces a norm on $P$ which is equivalent to the given norm, see \cite[Example 1.2.2]{SoTr}. If $P$ has a grading and a Real structure, we can make the inner product compatible with these extra structures. 
When equipped with such an inner product, $P$ becomes a Hilbert $\gA$-module, and the Banach algebras $\End_{\gA} (P)$ of $\gA$-linear continuous endomorphisms and $\Lin_{\gA} (P)$ coincide (since $P$ is finitely generated, there is even an equality $\Lin_{\gA}(P) = \Kom_{\gA}(P)$). The group $\Aut_{\gA} (P)$ of continuous $\gA$-linear automorphisms is an open subgroup in the Banach algebra $\Lin_{\gA}(P)$ and hence a Banach Lie group. 

Since $\Aut_{\gA}(P)$ is a Banach Lie group, it makes sense to talk about smooth fibre bundles $V \to M$ with fibre $P$ and structure group $\Aut_{\gA}(P)$ over a smooth manifold $M$; we call such bundles \emph{smooth bundles of projective finitely generated $\gA$-modules}.
By polar decomposition, the group $U(P)$ of unitary automorphisms is a strong deformation retract of $\Aut_{\gA} (P)$, and this has the consequence that each smooth bundle of finitely generated projective $\gA$-modules $V \to M$ admits a smooth fibrewise $\gA$-valued inner product.

An example of this structure arises when $G$ is a countable group and $Q \to M$ is a Galois cover of $M$ with group $G$ (the universal cover $Q=\tilde{M}$ is the prototypical example, with $G=\pi_1 (M)$). The group $G$ is a subgroup of the group $U (\gC^* (G))$ of unitary elements in the group $\cstar$-algebra (reduced or maximal). Therefore, it acts from the left on $\gC^* (G)$, and this action commutes with the right multiplication by $\gC^* (G)$. Thus we might form 
$\cL_Q:=Q \times_G \gC^* (G) \to M$, the \emph{Mishchenko-Fomenko line bundle}. 
The formula $\scpr{a,b} := a^* b$ defines a $\gC^* (G)$-valued inner product on the projective $\gC^* (G)$-right module $\gC^* (G)$, and $\scpr{\_,\_}$ is invariant under left-multiplication by elements $g \in G$.
Therefore, the bundle $\cL_Q$ has a canonical $\gC^* (G)$-valued inner product. 

Another more elementary example comes from a \emph{Spin structure} on a vector bundle $V\to X$ of rank $d$, which is given by a $\Spin (d)$-principal bundle $P \to X$ and an isometry $P \times_{\Spin (d)} \bR^d \cong V$. Recall that $\Spin (d)$ is a subgroup of the unitary subgroup of $\Cl^{d,0}$, so that we can form the spinor bundle $\spinor:=P \times_{\Spin (d) } \Cl^{d,0}$, where $\Spin (d)$ acts by left-multiplication. The formula $\scpr{x,y} := x^*y$ defines a $\Cl^{d,0}$-valued inner product on the fibres of $\spinor$. Moreover, there are a canonical Real structure and a grading on $\spinor$. 

\subsection{Differential operators linear over \texorpdfstring{$\cstar$}--algebras}\label{sec:diffops-ga-linear}

Let $V \to M$ be a smooth bundle of finitely generated projective $\gA$-modules on a smooth manifold. The space of compactly supported smooth sections of $V$ is denoted by $\Gamma_c (M;V)$. An $\gA$-linear map $D: \Gamma_c (M;V) \to \Gamma_c (M;V)$ is called \emph{differential operator of order $1$} if for each coordinate system $x: M \supset U \to \bR^d$ and each smooth local trivialization $h:V|_U \to U \times P$, there are smooth functions $b, a_1, \ldots, a_d : U \to \Lin_{\gA}(P)$ such that with respect to these local coordinates, $D$ has the form
\[
Ds (x)= b (x) s(x) + \sum_{j=1}^d a_j (x) \partial_{x_j} s(x). 
\]
The \emph{symbol} of $D$ is the map $T^* M \otimes V \to V$ defined using the commutator by $\smb_D (d_x f)s(x) :=i [D,f] s(x)$ (it only depends on $d_x f$ and $s(x)$). 
If $V$ is equipped with a $\gA$-valued inner product $\scpr{\_, \_}_{\gA}$ and $M$ carries a Riemannian metric with volume measure $\vol_M$, then $\Gamma_c (M;V)$ becomes a pre-Hilbert-$\gA$-module with the inner product
\begin{equation}\label{inner-product-on-space-of-cptsupp-section}
\scpr{s,t} := \int_M \scpr{s(x), t(x)} d \vol (x) \in \gA,
\end{equation}
and we define the Hilbert-$\gA$-module $L^2 (M;V)$ as the completion with respect to the norm induced by $\scpr{\_, \_}$. We say that $D$ is \emph{formally self-adjoint} if $\scpr{Ds,t} = \scpr{s,Dt}$ for all $s,t \in \Gamma_c (M;V)$. A \emph{Dirac operator} is a formally self-adjoint differential operator of order $1$ such that $\smb_D (df)^2 = - |df|^2$. If $V$ is equipped with a grading $\iota$, we only consider Dirac operators which are odd in the sense that $D \iota + \iota D=0$, and if there is a Real structure on $V$, we require that $D$ is real, i.e. $D \overline{s}= \overline{Ds}$ for all sections $s$.

A classical example arises if $T M$ has a spin structure with spinor bundle $\spinor_M$. In that case, the spin Dirac operator $\Dir: \Gamma_c (M; \spinor_M) \to \Gamma_c (M; \spinor_M)$ is a $\Cl^{d,0}$-linear operator \cite[\S II.7]{LM}.
If $M$ comes furthermore equipped with a map $\varphi: M \to BG$ to the classifying space of a countable discrete group $G$, we can pull back the Mishchenko-Fomenko line bundle on $BG$ (or first pull back the universal $G$-covering $EG \to BG$ to $M$ to a Galois covering $Q_{\varphi} \to M$ and then form $\cL_{Q_{\varphi}}=: \cL_{\varphi}$). On the tensor product $\spinor_M \otimes \cL_{\varphi}$, we have the Dirac operator $\Dir_{\cL_{\varphi}}$ (this construction is due to Rosenberg \cite{Ros}).

\subsection{Families of differential operators linear over \texorpdfstring{$\cstar$}--algebras}\label{sec:families-diffops-ga-linear}

Now we turn to \emph{families of differential operators}. Let $\pi: M \to X$ be a submersion with $d$-dimensional fibres $M_x := \pi^{-1} (x)$. Let $T_v M:= \ker (d \pi)\subset TM$ be the vertical tangent bundle of $M$. If $f: M \to \bR$ is a function, the fibrewise differential $d_v f$ of $f$ is the restriction of $df$ to $T_v M$. Let $V \to M$ be a smooth bundle of finitely generated projective $\gA$-modules, equipped with an $\gA$-valued inner product.  
A \emph{family of differential operators on $V$} is a differential operator $D$ on $V \to M$ such that $D$ commutes with multiplication by functions on $X$, i.e. $D (f \circ \pi) = (f \circ \pi) D$ for all smooth $f: X \to \bR$. For each $x \in X$ and $s\in \Gamma_c (M;V)$, the restriction $(Ds)|_{M_x}$ only depends on the restriction $s|_{M_x}$ of $s$. Thus, $D$ gives rise to a family $(D_x)_{x \in X}$, where $D_x$ is a differential operator on $\Gamma_c (M_x, V_x)$, and $(D_x)_{x\in X}$ determines $D$ uniquely. 

A \emph{fibrewise Riemannian metric} on $M$ is a smooth bundle metric $g$ on $T_v M$ (so in particular each fibre $M_x$ is a Riemannian manifold). If such a fibrewise Riemannian metric is given, we define $\gA$-valued scalar products as follows: for each $x \in X$ and $s,t \in \Gamma_c (M_x;V_x)$, we set
\begin{equation}\label{inner-product-for-sections-family}
\scpr{s,t}_x := \int_{M_x} \scpr{s(x), t(x)} d \vol_{M_x} (x) \in \gA.
\end{equation}
A family of differential operators is \emph{formally self-adjoint} if for each $x \in X$ and all sections $s,t \in \Gamma_c (M;V)$, we have $\scpr{Ds,t}_x = \scpr{s,Dt}_x$. If $D$ is formally self-adjoint and the \emph{fibrewise symbol} $\smb_D (d_v f) := i [D,f]$ satisfies $\smb_D (d_v f)^2 = - |d_v f|^2$, then we call $D$ a \emph{family of $\gA$-linear Dirac operators}.

\subsection{Unbounded operators on Hilbert modules}\label{subsec:unboundedops-hilbmod}

We have to deal with unbounded operators on Hilbert-$\gA$-modules. Our basic reference is \cite{Lance}, and we also recall results from \cite{KL}. The first definitions in the theory are parallel to the classical theory of unbounded operators on Hilbert spaces.

\begin{defn}\label{defn:unbounded-ooperator-hilbertmodule}
Let $E$ be a Hilbert $\gA$-module. A \emph{densely defined unbounded operator on $E$} is an $\gA$-linear operator $D: \dom (D) \to E$ which is defined on the dense $\gA$-submodule $\dom (D)\subset E$. We say that $D$ is \emph{symmetric} if $\scpr{u,Dv}=\scpr{Du,v}$ holds for all $u,v \in \dom (D)$. The \emph{graph norm} $\norm{\_}_D = \norm{\_}_{\gra}$ on $\dom (D)$ is the norm induced by the $\gA$-valued inner product $\scpr{u,v}_D := \scpr{u,v} + \scpr{Du,Dv}$. The graph $\gra (D):= \{(x,Dx) \vert x \in \dom (D)\} \subset E \oplus E$ with the induced inner product is isometric to $(\dom (D), \scpr{\_, \_}_D)$. 
We say that $D$ is \emph{closed} if $\gra (D) \subset E \oplus E$ is closed or equivalently if $(\dom (D), \scpr{\_,\_}_D)$ is complete, i.e. a Hilbert-$\gA$-module.
\end{defn}

There are estimates
\begin{equation}\label{estimate-unbounded-op1}
 \norm{u}^2 _D = \norm{\scpr{u,u}+\scpr{Du,Du}} \leq \norm{u}^2 + \norm{Du}^2
\end{equation}
and 
\begin{equation}\label{estimate-unbounded-op2}
  \norm{u}^2 _D \geq \frac{1}{2} (\norm{u}^2 + \norm{Du}^2).
\end{equation}
For the second one, note that $\scpr{u,u}_D = \scpr{u,u} + \scpr{Du,Du} \geq \scpr{u,u}$ implies $\norm{u}^2_{D} \geq \norm{u}^2$. Similarly $\norm{u}^2_D \geq \norm{Du}^2$, so that $ \norm{u}^2 _D \geq \max\{\norm{u}^2, \norm{Du}^2\} \geq \frac{ 1}{2} (\norm{u}^2+ \norm{Du}^2)$. Moreover, there is the useful equation
\begin{equation}\label{estimate-unbounded-op3}
 \scpr{u,u}_D = \scpr{(D+i)u,(D+i)u}
\end{equation}
for \emph{symmetric} $D$.

If $D$ is densely defined, the closure of $\gra(D)$ may or may not be the graph of a densely defined operator, and if it is, we say that $D$ is \emph{closeable}. A \emph{symmetric} densely defined operator is closeable and its closure is symmetric \cite[Lemma 2.1]{KL}.
A \emph{core} for a closed operator $D: \dom (D) \to E$ is a submodule $F \subset \dom (D)$ such that $D$ is the closure of $D|_F$.
When we construct an unbounded operator, we typically start with $D_0:\dom (D_0)\to E$ and then pass to the closure $D: \dom (D) \to E$. In that case, we call $\dom (D_0)$ the \emph{initial domain} of $D$; $\dom (D_0)$ is a core for $D$.

\begin{defn}\label{defn:adjoint-of-unbounded}
The \emph{adjoint} of a densely defined operator $D$ is the unbounded operator with domain
\[
\dom (D^*) = \{y \in E \vert \exists z \in E: \scpr{Dx,y} = \scpr{x,z} \forall x \in \dom (D)\},
\]
and for $y\in \dom (D^*)$, we let $D^* y:= z$ (the element $z$ uniquely determined because $\scpr{\_,\_}$ is definite). It is not necessarily true in general that $\dom (D^*) \subset E$ is dense, but if $D$ is symmetric, then $\dom (D) \subset \dom (D^*)$ and hence $D^*$ is densely defined. 
One says that $D$ is \emph{self-adjoint} if $D=D^*$.
\end{defn}

The adjoint of any densely defined operator is closed. A closed symmetric operator is \emph{regular} if\footnote{Compositions of unbounded operators need to be treated with care, see \cite[p. 95]{Lance}.} $1+D^* D$ has dense range. The class of unbounded operators which are reasonably well-behaved are the regular self-adjoint operators. These have a spectral-theoretic characterization.

\begin{prop}\cite[Proposition 4.1]{KL}\label{spectral-criterion}
For a closed, symmetric, densely defined operator $D$ on a Hilbert $\gA$-module $E$, the following are equivalent:
\begin{enumerate}
\item $D$ is regular and self-adjoint.
\item There is $\lambda \in \bC \setminus \bR$ such that the operators $(D+ \lambda), (D-\lambda): \dom (D) \to E$ are invertible.
\item For all $\lambda \in \bC \setminus \bR$, the operator $(D-\lambda):\dom (D) \to E$ is invertible.
\end{enumerate}
\end{prop}

If $D: \dom (D) \to E$ is symmetric, $u\in \dom (D)$ and $\lambda \in \bC$, then
\[
\scpr{(D+\lambda)u, (D+\lambda)u} = \scpr{(D+\Re(\lambda)) u,(D+\Re(\lambda)) u} + |\Im (\lambda)|^2 \scpr{u,u}
\]
and we obtain, as in the classical situation, the inequality
\begin{equation}\label{coercivity}
\norm{(D+\lambda) u}^2 \geq |\Im (\lambda)|^2 \norm{u}^2.
\end{equation}

\subsection{Self-adjointness of \texorpdfstring{$\gA$}--linear differential operators}\label{subsec:proof-chernoof}

We return to the setting introduced in section \ref{sec:diffops-ga-linear}: $V \to M$ is a bundle of finitely generated projective $\gA$-modules with inner product on a Riemannian manifold and $D: \Gamma_c (M;V) \to \Gamma_c (M;V)$ is a formally self-adjoint differential operator of order $1$. Then $D$ defines an unbounded operator on $L^2 (M;V)$ with initial domain $\Gamma_c (M;V)$.  
We want to prove that the closure of $D$, denoted\footnote{We are slightly abusing notation by denoting three different objects by the same letter $D$. When the distinction becomes crucial, we will introduce more precise notation.} $D: \dom (D) \to L^2 (M;V)$, is self-adjoint and regular, under suitable hypotheses. 

\begin{defn}\label{defn:coercive-function}
A \emph{coercive function} on $M$ is a proper smooth function $f: M \to \bR$ which is bounded from below.
\end{defn}

\begin{defn}\label{defn:complete-operator-on-mfd}
Let $M, V,D$ be as before. We say that $(M,D)$ is \emph{complete} if there exists a coercive function $f: M \to \bR$ such that the commutator $[D,f]$ is bounded.
\end{defn}

In \cite{HR}, \S 10.2, this condition is called ``$M$ is complete for $D$''. The origin of the word is that a Dirac operator on a complete Riemannian manifold is complete in the above sense of the word \cite[p. 623]{Wolf}.

\begin{thm}\label{chernoff-theorem}
If $(M,D)$ is complete, then $D: \dom (D) \to L^2 (M;V)$ is self-adjoint and regular.
\end{thm}

The case $\gA=\bC$ is classical and due to Chernoff \cite{Chernoff} and Wolf \cite{Wolf}; see also \cite[Proposition 10.2.10]{HR}. For arbitrary $\gA$, Hanke, Pape and Schick \cite[Theorem 1.5]{HPS} proved a result which is only slighly less general than Theorem \ref{chernoff-theorem}. 
We give an alternative proof, using a different method. 

\subsubsection*{Localization of Hilbert modules}
For the proof of Theorem \ref{chernoff-theorem}, we use a result by Kaad and Lesch \cite{KL} that characterizes self-adjoint regular operators. 
This involves the \emph{localization} of a densely defined operator $D$ at a $*$-representation $\pi: \gA \to \Lin (H_{\pi})$ of $\gA$ on a Hilbert space (it is not required that $H_{\pi}$ is separable), introduced in \cite[\S 2.4]{KL}. 
The algebraic tensor product $E \odot_{\gA} H_{\pi}$ is a $\gC$-vector space and has the inner product 
\[
\scpr{x \otimes v, y \otimes w}:= \scpr{v, \pi (\scpr{x,y}_{\gA}) w} \in \gC,
\]
which is positive semidefinite by \cite[Prop 4.5]{Lance}
. Define the Hilbert space $E^{\pi}=E \otimes_{\gA} H_{\pi}$ as the completion of $E \odot_{\gA} H_{\pi}$ with respect to that scalar product. For $F \in \Lin_{\gA} (E)$, the induced map $F_{\pi}=F \otimes 1: E \otimes_{\gA} H_{\pi} \to E \otimes_{\gA} H_{\pi}$ is bounded \cite[p. 42]{Lance}. This gives an associated representation $\tilde{\pi}:\Lin_{\gA}(E) \to \Lin (E^{\pi})$ of $\cstar$-algebras.

Let $D: \dom (D) \to E$ be a closed, densely defined and symmetric operator on a Hilbert $\gA$-module. Let $\dom (D^{\pi}_0):= \dom (D) \odot_{\gA} H_{\pi} \subset E^{\pi}$ and define $D^{\pi}_{0}: \dom (D^{\pi}_0)\to E^{\pi}$ by the formula
\[
D_0^{\pi} (x \otimes v) := Dx \otimes v.
\]
By \cite[Lemma 2.5]{KL}, $D_0^{\pi}$ is densely defined and symmetric, and the \emph{localization} $D^{\pi}$ is by definition the closure of $D^{\pi}_0$.

\begin{thm}[Kaad and Lesch, \cite{KL} Theorem 4.2]\label{thm:kaadlesch}
The following conditions on a densely defined symmetric and closed operator $D$ on a Hilbert $\gA$-module $E$ are equivalent.
\begin{enumerate}
\item $D$ is self-adjoint and regular.
\item For each cyclic representation $\pi$ of $\gA$, the localization $D^{\pi}$ is a self-adjoint operator on the Hilbert space $E^{\pi}$.
\end{enumerate}
\end{thm}

Recall that a \emph{cyclic representation} of $\gA$ is a $*$-representation $\pi: \gA \to \Lin (H_{\pi})$ on a Hilbert space such that there exists $h_0 \in H_{\pi}$ so that $\gA \to H_{\pi}$, $a \mapsto \pi(a) h_0$ has dense image.
The route for the proof of Theorem \ref{chernoff-theorem} is now predictable: we have to describe the localization $D^{\pi}$ in a precise way and we have to show that it is self-adjoint (this step is a variation of one of the proofs in the Hilbert space case).

Let $P$ be the fibre of the bundle $V \to M$. There is a $U(P)$-principal bundle $Q \to M$ and an isometry $Q \times_{U(P)} P \cong V$. The representation $\pi$ induces a continuous group homomorphism (even smooth) $U(P) \to U(P^{\pi})$ to the unitary group of the Hilbert space $P^{\pi}$ ($P^{\pi}$ is typically \emph{not} finite-dimensional). We obtain a bundle
\[
V^{\pi} := Q \times_{U(P)} P^{\pi} \to M
\]
with fibre the Hilbert space $P^{\pi}$ and structure group the unitary group of $P^{\pi}$ (with the norm topology). There is a natural $\gC$-valued inner product on $\Gamma_c (M;V^{\pi})$, and the completion is a Hilbert space $L^2 (M;V^{\pi})$ (not surprisingly, $L^2 (M;V^{\pi})$ and $L^2 (M;V)^{\pi}$ are isometric, see Lemma \ref{lem:localization-explicit} below).

There is a localized version of the differential operator $D$. In local coordinates, $D$ is given by the formula $Ds (x)= b (x) s(x) + \sum_{j=1}^d a_j (x) \partial_{x_j} s(x)$, and we set  
\[
D_{\pi}s(x) := \tilde{\pi}(b(x)) s(x) + \sum_{j=1}^d \tilde{\pi}(a_j (x)) \partial_{x_j} s(x).
\]
This is a differential operator on the Hilbert bundle $V^{\pi}$. A priori, it is not clear how $D_{\pi}$ relates to the Hilbert module localization $D^{\pi}$ (therefore, the notational difference). Before we clarify the relation between $D^{\pi}$ and $D_{\pi}$, we prove that $D_{\pi}$ is essentially self-adjoint (in the sense of classical unbounded operator theory).

Using the coordinate expression for $D_{\pi}$, it is easy to see that $D_{\pi}: \Gamma_c (M; V^{\pi}) \to \Gamma_c (M; V^{\pi})$ is formally self-adjoint and hence $D_{\pi}: \Gamma_c (M;V^{\pi})\to L^2 (M;V^{\pi})$ is symmetric. Moreover, if $f: M \to \bR$ is a function, then 
\[
[D_{\pi},f] =  \sum_{j=1}^d \tilde{\pi}(a_j (x)) (\partial_{x_j}f) ,
\]
and since $\tilde{\pi}$ is bounded, this proves that $[D_{\pi},f]$ is bounded once $[D,f]$ is bounded. Therefore, the next result implies that $D_{\pi}$ is essentially self-adjoint.

\begin{prop}\label{chernoff-classical}
Let $E \to M$ be a smooth Hilbert bundle on $M$ (unitary group as structure group) and let $D: \Gamma_c (M;E) \to \Gamma_c (M;E)$ be a formally self-adjoint differential operator of order $1$. Assume that there is a coercive function $f: M \to \bR$ such that $[D,f]$ is a bounded operator. Then the closure $\overline{D}$ of the operator $D: \Gamma_c (M;E) \to L^2 (M;E)$ is self-adjoint.
\end{prop}

For Hilbert bundles $E$ of \emph{finite rank}, this is the classical Chernoff theorem. Some proofs we found in the literature use arguments which are only valid if $E$ has finite rank (namely, that a Friedrichs mollifier is a compact operator).
The proof of \cite[Proposition 10.2.10]{HR} essentially works in full generality, and we now go through this proof to demonstrate this.

\begin{proof}
Let $\Gamma := \Gamma_c (M;E)$ be the initial domain and $H:= L^2 (M;E)$. 
We have to prove that the domain $\dom (D^*)$ of the adjoint $D^*$ is contained in $\dom (\overline{D})$. So let $u \in H$ be in the domain of $D^*$. By definition, this means that there is $C_0 \geq 0$ with
\[
|\scpr{Dv, u}| \leq C_0 \norm{v}
\]
for all $v \in \Gamma$. Assume that $u$ has compact support to start with. Fix a family $(F_t)_{t\in (0,1]}$ of Friedrichs mollifiers. 
This is a family of bounded self-adjoint operators on $H$ such that $F_t$ and $[D,F_t]$ is uniformly bounded, say $\norm{[D,F_t]} \leq C_1$ and $\norm{F_t} \leq C_2$. Moreover, $\norm{F_t u-u} \to 0$ and $F_t u \in \Gamma$. Let $t_n \to 0$ and $F_n := F_{t_n}$. Then, for all $v \in \Gamma$,
\[
|\scpr{DF_n u, v}| = |\scpr{u,F_n D v} |\leq | \scpr{u, [F_n,D] v}| + |\scpr{u, D F_n v}| \leq C_1\norm{u} \norm{v}  + C_0 C_2\norm{v}=: C \norm{v}.
\]
Since $v$ was arbitrary, $\norm{DF_n u } \leq C$, independent of $n$. Thus the sequence $u_n :=F_n u \in \Gamma$ converges to $u$, and $Du_n$ is uniformly (in $n$) bounded. Hence by the sequential Banach-Alaoglu Theorem \cite[Theorem III.7.3]{Werner}, we can assume that $Du_n$ is weakly convergent, after passing to a subsequence. 
By the Banach-Saks Theorem \cite[Satz V.3.8]{Werner}, there is a further subsequence of $u_n$ such that the sequence of arithmetic means $v_n:=\frac{1}{m}\sum_{n=1}^{m} Du_n$ is convergent in norm. Let $w_n := \frac{1}{m}\sum_{n=1}^{m} u_n$, so that $Dw_n =v_n$ and $w_n \to u$. Therefore, the sequence $(w_n, Dw_n)$ in the graph of $D$ is convergent, which proves that $u=\lim_n w_n \in \dom (\overline{D})$. So far, this argument shows that any compactly supported element $u$ of the domain of $D^*$ lies in $\dom (\overline{D})$.

Now let $u \in \dom (D^*)$ be arbitrary. Let $h_n: \bR \to [0,1]$ be a sequence of smooth functions with $|h_n'(t)| \leq \frac{1}{n}$, $h_n(t)=1$ for $t \leq n$ and $h_n (t) =0$ for $t \geq 3n$. Let $g_n := h_n \circ f$, with the function $f: M \to \bR$ guaranteed by the hypothesis of the Theorem. Then $\norm{g_n u - u} \to 0$, and $\norm{[D,g_n]} \leq \frac{1}{n} \norm{[D,f]} \to 0$. Because $u \in \dom (D^*)$, the estimate
\begin{equation}\label{est-chrnoff-proof1}
|\scpr{Dv, g_n u}| \leq |\scpr{[g_n, D]v,  u}|+|\scpr{D(g_n v), u}| \leq \frac{1}{n} \norm{[D,f]} \norm{u} \norm{v} + C_0 \norm{g_n v}  \leq C \norm{v}
\end{equation}
holds for all $v \in \Gamma$, with $C:=\frac{1}{n} \norm{[D,f]} \norm{u}  + C_0$. This proves $g_n u \in \dom (D^*)$, but $g_n$ has compact support and so by the first part of the proof, $g_n u \in \dom (\overline{D})$. Next, we claim that the sequence $D(g_n u)$ is uniformly bounded. To see this, let $v \in \Gamma$ and use that
\[
|\scpr{D g_n u,v}| = |\scpr{ g_n u, D v}| \leq C \norm{v}
\]
by (\ref{est-chrnoff-proof1}). Again this holds for all $v \in \Gamma$, but $\Gamma$ is dense in $H$ and so $\norm{D g_n u} \leq C$. The same argument as in the first part of the proof then shows that $u \in \dom (\overline{D})$. 
\end{proof}

To finish the proof of Theorem \ref{chernoff-theorem}, we have to relate $D_{\pi}$ and $D^{\pi}$. 

\begin{lem}\label{lem:localization-explicit}
There is a canonical isometry 
\[
\Phi:(L^2 (M;V) )^{\pi}\to L^2 (M;V^{\pi})
\]
of Hilbert spaces. The relation $\Phi \circ (D_0^\pi) = D_{\pi} \circ \Phi$ holds on $\Gamma_c (M;V) \odot_{\gA} H_{\pi}$. 
Provided that $\pi$ is cyclic, $\Phi$ takes the algebraic tensor product $\Gamma_c (M;V) \odot_{\gA} H_{\pi}$ onto a core of $D_{\pi}$.
\end{lem}

\begin{proof}
Define
\[
\Phi: \Gamma_c (M;V) \odot_{\gA} H_{\pi} \to L^2 (M; V^{\pi})
\]
by $s \otimes_{\gA} h \mapsto (x \mapsto s(x) \otimes_{\gA} h \in V_x^{\pi})$. It is clear that $\Phi (\Gamma_c (M;V) \odot_{\gA} H_{\pi}) \subset \Gamma_c (M; V^{\pi})$, and it is also clear that $\Phi$ intertwines $D_{\pi}$ and $D_0^\pi$ (look at a local trivialization). Next, we claim that $\Phi$ preserves the inner product. To see this, compute on elementary tensors
\[
\scpr{\Phi(s \otimes h), \Phi(s' \otimes h')}_{L^2 (M;V^{\pi})} = \int_M \scpr{\Phi (s\otimes h)(x), \Phi (s'\otimes h')(x)}_{V_x^{\pi}} d\vol (x)= 
\]
\[
\int_M \scpr{s(x) \otimes h,s'(x) \otimes h' }_{V_x^{\pi}}d \vol (x)= \int_M \scpr{ h,\pi(\scpr{s(x),s'(x)}_{V_x} ) h' }_{H_{\pi}} d\vol (x)
\]
and 
\[
\scpr{s \otimes h, s' \otimes h'}_{L^2 (M;V)^{\pi}} = \scpr{h, \pi (\scpr{s,s'}_{\gA}) h'}_{H_{\pi}} = \scpr{h, \pi (\int_M \scpr{s(x),s'(x)}_{\gA} d\vol (x) ) h'}_{H_{\pi}}.
\]
Since the integral commutes with bounded linear maps, both inner products are equal. 
So $\Phi$ preserves the inner product, and completes to an isometry (we have not yet seen that it is surjective). For the last statement of the Lemma, we have to prove that each $u \in \dom (D_{\pi})$ is the limit (in the graph norm) of a sequence of elements of $\Phi (\Gamma_c (M;V) \odot_{\gA} H_{\pi})$. If this is done, $\Phi$ has dense range and is therefore a surjective isometry.

The initial domain of $D_{\pi}$ is $\Gamma_c (M;V^{\pi})$, and it is enough to prove that each $u \in \Gamma_c (M;V^{\pi})$ can be approximated in the graph norm by elements of $\Phi (\Gamma_c (M;V) \odot_{\gA} H_{\pi})$. Since the support of $u$ has finite measure, it is enough to approximate $u$ in the $C^1$-norm.
By an argument with a partition of unity, it is enough to do this when $u$ is supported in a coordinate patch, where $V$ is trivial. So without loss of generality $M=\bR^d$ and $V= \bR^d \times P$. But $P$ is finitely generated projective, hence a direct summand of $\gA^n$, and so $P^{\pi}$ is a direct summand of $H_{\pi}^n$. Using the projection maps, we further reduce the problem to the case when $P=\gA^n$ and finally $P=\gA$. 
In that case $P^{\pi}=H_{\pi}$, and $\Phi$ is the map $C^{\infty}_c (\bR^d; \gA) \odot_{\gA} H_{\pi} \to L^2 (\bR^d; H_{\pi})$, $\Phi (s \otimes h)= (x \mapsto \pi (s(x)) h)$.

After all these reductions, we face the following problem. Assume that $u:\bR^d \to H_{\pi}$ is a smooth function with compact support. Then there exists a sequence $u_n$ of functions in the image of $\Phi$ which converges uniformly in the $C^1$-norm to $u$. 

The image $K:=u(\bR^d) \cup \partial_1 u (\bR^d) \cup \ldots \cup \partial_d u(\bR^d)$ is a compact subset of $H_{\pi}$ and therefore, it is contained in a separable closed subspace $H_K \subset H_{\pi}$. Pick an increasing sequence of finite rank projections $p_n$ on $H_K$ which tends strongly to the identity. By the Arzel\`a-Ascoli theorem, $p_n|_K$ is uniformly convergent.
Hence $p_n \circ u$ converges to $u$ in the $C^1$-norm. It is now enough to approximate each $p_n u$, and we have reduced the problem to a function $u$ with values in a fixed finite-dimensional subspace $H_0 \subset H_{\pi}$. Let $(v_1, \ldots, v_r)$ be an orthonormal basis of $H_0$ and let $\eps>0$. Choose a cyclic vector $h_0 \in H_\pi$ and pick $a_1, \ldots, a_r \in \gA$ with $\norm{\pi (a_i)h_0 - v_i} \leq \frac{\eps}{\sqrt{r}}$ (it is at this place where we use that $\pi$ is cyclic). Let $\sigma: H_0 \to \gA$ be the linear map given by $\sigma (v_i):= a_i$. The operator norm of the linear map $H_0\to H_{\pi}$, $v \mapsto v- \pi(\sigma (v)) h_0$ is at most $\epsilon$. Define $s: \bR^d \to \gA$ by $s(x):=\sigma (u(x))$. It follows that $\norm{u- \Phi(s\otimes h_0)}_{C^1} \leq (d+1)\epsilon$, and hence we can approximate $u$ by elements of $\im (\Phi)$.
\end{proof}

\begin{proof}[Proof of Theorem \ref{chernoff-theorem}]
By Theorem \ref{thm:kaadlesch}, we have to show that the Hilbert module localizations $D^{\pi}$ are self-adjoint, for each cyclic representation.
By construction, $\Gamma_c (M;V)\subset \dom (D)$ is a core for $D$. By \cite[Theorem 3.3]{KL}, the algebraic tensor product $\Gamma_c (M;V) \odot_{\gA} H_{\pi}$ is a core for the localization $D^{\pi}$. By Lemma \ref{lem:localization-explicit}, this core goes to a core for $D_{\pi}$ under the isometry $\Phi$. Therefore, the closure of $D^{\pi}_0: \Gamma_c (M;V) \odot_{\gA} H_{\pi} \to L^2 (M; V)^{\pi}$ corresponds to the closure of $D_{\pi} : \Gamma_c (M;V^{\pi}) \to L^2 (M; V^{\pi})$ under $\Phi$. By Proposition \ref{chernoff-classical}, the closure of $D_{\pi}$ is self-adjoint, and hence so is the closure of $D^{\pi}_0$, which is nothing else than $D^{\pi}$.
\end{proof}

\subsection{Functional calculus for unbounded operators}\label{subsec:FunctCalUnbounded}

Next, we recall the functional calculus of unbounded self-adjoint operators on Hilbert modules. Parts of these result can be found in \cite{Lance} and some are due to Kucerovsky \cite{Kuc}. Let $E$ be a Hilbert-$\gA$-module and let $D: \dom (D) \to E$ be a self-adjoint and regular operator. The \emph{spectrum} of $D$ is the closed subset $\spec (D)\subset \bC$ of all $\lambda\in \bC$ such that $(D-\lambda): \dom (D) \to E$ is not invertible. Proposition \ref{spectral-criterion} shows that $\spec (D) \subset \bR$. 
Let $\gC (\overline{\bR})$ be the $\cstar$-algebra of continuous functions $f:\bR \to \gC$ such that the limits $\lim_{t \to \pm \infty} f(t) \in \gC$ exist (this can be considered as the $\cstar$-algebra of all continuous functions on the extended real line $\overline{\bR}=[-\infty,\infty]$). It contains $\gC_0 (\bR)$ as an ideal of codimension $2$.

\begin{thm}\label{thm:kucerovsky}
Let $D$ be self-adjoint regular operator on a Hilbert $\gA$-module $E$. Then there is a unital $*$-homomorphism
\[
\Phi=\Phi_D: \gC(\overline{\bR}) \to \Lin_{\gA} (E)
\]
with the following properties:
\begin{enumerate}
\item $\Phi_D (\frac{1}{t\pm i}) = (D\pm i)^{-1}$,
\item $\norm{\Phi_D(f)} \leq \norm{f}_{C^0}$. 
\item Let $f,f_n \in \gC(\overline{\bR})$. Assume that $\norm{f_n} \leq C$ for all $n$ and that $f_n \to f$ uniformly on compact subsets of $\bR$. Then 
\[
\Phi_D(f_n)u \to \Phi_D(f)u
\]
for all $u \in E$.
\item If $f \in \gC_0(\bR)$ is such that $g(t)=tf(t) \in \gC(\overline{\bR})$, then $\Phi_D (f)$ maps $E$ into $\dom (D)$, and we have $D\Phi_D(f)=\Phi_D(f) D=\Phi_D(g)$.
\item Let $\overline{\spec (D)} \subset \overline{\bR}$ be the closure of the spectrum in $\overline{\bR}$. For $f \in \gC(\overline{\bR})$, one has 
\[
\spec_{\Lin_{\gA}(E)} (\Phi_D (f)) = f(\overline{\spec(D)}).
\]
\item $\norm{\Phi_D(f)} \leq \norm{f}_{\spec (D)} := \sup_{t \in \spec (D)} |f(t)|$ (in particular, if $f|_{\spec(D)}=0$, then $\Phi_D (f)=0$ and so $\Phi_D (f)$ is well-defined once $f$ is a function only defined on $\spec(D)$). 
\item If $D$ is Real and $f$ real-valued, then $f(D)$ is Real. Let $\iota$ be a grading on $E$, and $D\iota + \iota D=0$. Then for even $f$, $f(D)\iota =\iota f(D)$, and for odd $f$, $f(D) \iota=-\iota f(D)$.
\end{enumerate}
\end{thm}

\begin{proof}
The existence of a $*$-homomorphism $\Phi_D$ with the property (1) is proven in \cite[Theorem 10.9]{Lance}. Part (2) is a general property of $*$-homomorphisms. Part (3) is a special case of \cite[Lemma 13]{Kuc}; for an alternative proof see also \cite[Lemma 9.9]{Stolz2}. Part (4) follows from \cite[Proposition 10.4]{Lance}. For part (5), let $\gI:= \ker (\Phi_D) $, an ideal in $\gC(\overline{\bR})$. There is a closed set $S \subset \overline{\bR}$ with $\gI=\{f \vert f|_S=0\}$, and $\gC(\overline{\bR}) /\gI=\gC (S)$. Then $\Phi_D$ factors as $\gC(\overline{\bR}) \stackrel{\rho}{\to} \gC(S) \stackrel{\Phi'}{\to} \Lin_{\gA}(E)$, $\rho$ is the restriction map and $\Phi'$ is injective. By \cite[Theorem VIII.4.8]{Conway} and its proof, $\Phi'$ is an isometry and for each $f \in \gC(\overline{\bR})$, we have $f(S)=\spec_{\gC(S)} (\rho(f)) = \spec_{\Lin_{\gA}(E)} (\Phi_D(f))$. Once we know that $S=\overline{\spec (D)}$, claims (5) and (6) follow immediately. To prove this, let $\lambda \in \bR$ and observe
\[
\begin{split}
\lambda \not \in \spec (D)\, \Leftrightarrow\\
(D-\lambda)(D+i)^{-1} = \Phi_D (\frac{t-\lambda}{t+i})= \Phi' (\rho (\frac{t-\lambda}{t+i})) \text{ is invertible in }\Lin_{\gA}(E) \,\Leftrightarrow\\
0 \not \in \spec_{\Lin_{\gA}(E)} (\Phi' (\rho (\frac{t-\lambda}{t+i})))\, \Leftrightarrow\\
0 \not \in \spec_{\gC(S)}  (\rho (\frac{t-\lambda}{t+i}))\,\Leftrightarrow\\
\lambda \not \in S.
\end{split}
\]
Part (7) is clear. 
\end{proof}

From now on, we shall write $f(D):= \Phi_D(f)$, and a symbol like $\frac{1}{(1+D^2)^{1/2}}$ denotes the operator $\Phi_D (f)$, $f(x):= \frac{1}{(1+x^2)^{1/2}}$. So far, the tools we introduced are geared to the study of differential operators of order $1$. Sometimes, we need to consider operators of order $2$, such as squares $D^2$, but also operators of the form $D^2+A$, for some order $0$ operator $A$, are relevant for us. 

\begin{prop}\label{prop:square-of-selfadjoint}
Let $D$ be a self-adjoint and regular operator on a Hilbert $\gA$-module $H$. The operator $D^2$ with domain $\dom (D^2) := \{ u \in \dom (D) \vert D u \in \dom (D)\}$ is densely defined, closed, self-adjoint and regular, and $\spec (D^2) \subset [0,\infty)$. 
\end{prop}

\begin{proof}
By Theorem \ref{thm:kucerovsky}, $(D- i)^{-1}(H)\subset \dom (D)$. Therefore, if $u \in \dom (D)$, then $(D-i)^{-1}u \in \dom (D)$ and $D(D-i)^{-1} u= (D-i)^{-1} Du \in \dom (D)$. Therefore $(D-i)^{-1} (\dom (D)) \subset \dom (D^2)$. Since the bounded operator $(D- i)^{-1}: H \to \dom (D) \subset H$ has dense range and hence carries the dense submodule $\dom (D)$ to a dense submodule of $H$, it follows that $\dom (D^2)$ is dense. Moreover $(D-i): \dom (D^2) \to \dom (D)$ is bijective. 

For $u \in \dom (D^2)$, we compute the graph scalar product, using the symmetry of $D$, as
\[
\scpr{(D-i)u,(D-i)u}_D = \scpr{(D-i)u,(D-i)u} + \scpr{D(D-i)u,D(D-i)u} = 2\scpr{Du, Du} +\scpr{u,u}+ \scpr{D^2 u,D^2 u}.
\]
It follows that 
\[
\scpr{(D-i)u,(D-i)u}_D \geq \scpr{u,u}_{D^2} \Rightarrow \norm{(D-i)u}_D \geq \norm{u}_{D^2}.
\]
Moreover
\[
\norm{(D-i)u}_D^2 \leq \norm{u}_{D^2}^2 + 2 \norm{\scpr{Du,Du}} = \norm{u}_{D^2}^2 + 2 \norm{\scpr{D^2 u,u}} \leq \norm{u}_{D^2}^2 + 2 \norm{u} \norm{D^2 u}.
\]
But
\[
\norm{u} \norm{D^2 u} \leq \norm{u}_{D^2}^2,
\]
so that $\norm{(D-i)u}_D^2 \leq 3 \norm{u}_{D^2}^2$.
Therefore, the bijection $(D-i):\dom (D^2)\to \dom (D)$ is a topological isomorphism onto its image. Hence $D^2$ is closed iff $D$ is.

If $\lambda \in \bC \setminus [0, \infty)$, pick $\mu\in \bC \setminus \bR$ with $\mu^2 =\lambda$. Then 
\[
(D^2-\mu^2) = (D-\mu)(D+\mu).
\]
Since $(D\pm \mu)$ are both invertible, so is $D^2-\mu^2$. To finish the proof, apply Proposition \ref{spectral-criterion}. 
\end{proof}

There is a useful integral formula for the operator $\frac{1}{(1+D)^{1/2}}$, when $D$ is self-adjoint. We state in a slightly more general form. If $D$ is self-adjoint, then the operator $L:=1+D^2$ is self-adjoint and \emph{strictly positive}, in other words, there is $c>0$ with $\scpr{Lu,u} \geq c^2\scpr{u,u}$ for all $u$ (here of course $c=1$). 

\begin{prop}\label{prop:positive-operator-of-order-two}
Let $L$ be a self-adjoint regular operator on a Hilbert $\gA$-module $E$ and assume that $L$ is strictly positive, i.e. $\scpr{Lu,u} \geq c \scpr{u,u}$ for some $c>0$. Then $L+t$ is invertible for all $t \geq 0$, we have $\norm{(L+t)^{-1} } \leq \frac{1}{t+c}$, $L^{-1}$ is a positive element of $\Lin_{\gA} (E)$ and there is the absolutely convergent integral representation
\[
\frac{1}{\sqrt{L}} = \frac{2}{\pi} \int_0^{\infty} \frac{1}{L+t^2} dt. 
\]
\end{prop}

\begin{proof}
First we show that $L$ is invertible. By (\ref{cauchy-schwarz})
\[
\norm{Lu} \norm{u} \geq \norm{\scpr{Lu,u}} \geq c \norm{\scpr{u,u}}=c \norm{u}^2
\]
and hence $\norm{Lu} \geq c \norm{u}$ for each $u \in \dom (L)$. Since $L$ is self-adjoint and regular, $L+is$ is invertible for all $s \in \bR$, $s\neq 0$. Then
\[
\norm{(L+is)u}^2  = \norm{\scpr{Lu,Lu}+ s^2 \scpr{u,u}}\geq \norm{\scpr{Lu,Lu}} - s^2 \norm{\scpr{u,u}} = 
\]
\[
\norm{Lu}^2 - s^2 \norm{u}^2 \geq (c^2-s^2) \norm{u}^2
\]
and $\norm{(L+is)^{-1}} \leq \frac{1}{\sqrt{c^2-s^2}}$ for $0<|s|<c$. When combined with the equation
\[
\norm{(L+is)^{-1}} = \sup_{t \in \spec (L)} |\frac{1}{t+is}| = \frac{1}{\inf_{t\in \spec (L)} |t+is|}
\]
derived from Theorem \ref{thm:kucerovsky} (6), it follows that
\[
\inf_{t\in \spec (L)} |t+is| \geq \sqrt{c^2-s^2}.
\]
In the limit $s \to 0$, we get that $(-c,c) \cap \spec (L)= \emptyset$; in particular $0 \not \in \spec (L)$ and $L$ is invertible.

Now let $U = \bR\setminus \spec (L)$ and assume that $-t \in U$, $t>0$. Then $L+t$ is invertible and by (\ref{cauchy-schwarz})
\[
\norm{(L+t)u} \norm{u} \geq \norm{\scpr{(L+t)u,u}} \geq (c+t) \norm{u}^2, 
\]
hence $\norm{(L+t)u} \geq (c+t) \norm{u}$ for all $u \in \dom (L)$. Running the same argument as before, it follows that $(-c-t,c+t)\cap \spec (L+t)=\emptyset$, hence $(-2t-c,c)\subset U$. Therefore if $-t \in U$, then $(-2t-c,c)\subset U$, which implies that $(-\infty,c) \subset U$. Therefore $\spec_{\Lin_{\gA}(E)} (L^{-1}) \subset [0, \frac{1}{c}]$ by Theorem \ref{thm:kucerovsky} (6), and $L^{-1}$ is positive. To prove the integral formula, we define $T:= L^{-1} \in \Lin_{\gA} (E)$ and rewrite the formula as
\begin{equation}\label{integral-formula-second-form}
\sqrt{T}= \frac{2}{\pi} \int_0^{\infty} \frac{1}{L+t^2} dt = \frac{2}{\pi} \int_0^{\infty} \frac{T}{1+Tt^2} dt.
\end{equation}
This formula only involves \emph{elements of the $\cstar$-algebra $\Lin_{\gA}(E)$}. We claim that (\ref{integral-formula-second-form}) holds in every $\cstar$-algebra. In other words, let $\gB$ be a $\cstar$-algebra and $T \in \gB$ a positive self-adjoint element. Then the formula (\ref{integral-formula-second-form}) holds and the integral on the right-hand side converges absolutely. To show that the integral converges absolutely, note that $\norm{\frac{T}{1+Tt^2}} \leq \sup_{0 \leq x \leq \norm{T}} \frac{x}{1+t^2 x} =\frac{\norm{T}}{1+t^2 \norm{T}}$.

For $\gB= \Lin (H)$, $H$ a Hilbert space, the formula (\ref{integral-formula-second-form}) is proven using spectral measures. Let $v \in H$ and let $\mu_v$ be the unique Borel measure on $\spec (T)$ such that $\scpr{f(T)v,v} = \int_{\spec(T)} f(x) d\mu_v (x)$ for all continuous $f: \spec (T) \to \bC$. Using the elementary integral $\frac{2}{\pi}\int_0^{\infty} \frac{x}{1+x t^2} dt = \sqrt{x}$, we compute by Fubini's theorem
\[
\scpr{\sqrt{T}v,v} = \int_{\spec(T)} \sqrt{x} d\mu_v (x)= \int_{\spec(T)} \frac{2}{\pi}\int_0^{\infty} \frac{x}{1+x t^2} dt d\mu_v (x) = 
\]
\[
= \frac{2}{\pi}\int_0^{\infty} \int_{\spec(T)} \frac{x}{1+x t^2}  d\mu_v (x) dt= \frac{2}{\pi}\int_0^{\infty} \scpr{ \frac{T}{1+T t^2}v,v} dt = \scpr{ \frac{2}{\pi}\int_0^{\infty}  \frac{T}{1+T t^2} dt v, v}.
\]
Since this holds for all $v\in H$ and since both operators are self-adjoint, this proves (\ref{integral-formula-second-form}) for $\gB=\Lin (H)$. 

In the general case, let $\Pi:\gB \to \Lin (H)$ be a $*$-representation on some Hilbert space. Then we get
\[
\Pi (\sqrt{T}) = \sqrt{\Pi (T)} =  \frac{2}{\pi} \int_0^{\infty} \frac{\Pi(T)}{1+\Pi(T)t^2} dt = \Pi( \frac{2}{\pi} \int_0^{\infty} \frac{T}{1+Tt^2} dt)
\]
because the functional calculus commutes with $*$-homomorphisms. If we use an injective $\Pi$, the desired formula follows, and the Gelfand-Naimark-Segal representation theorem guarantees the existence of a faithful representation.
\end{proof}

\section{Parametrized elliptic regularity theory}\label{sec:parametrized-regularity}

The next goal is to setup a formalism in which elliptic regularity for \emph{families} of elliptic differential operators as introduced in section \ref{sec:families-diffops-ga-linear} can be formulated and proven. We use the theory of \emph{continuous fields of Banach spaces} due to Dixmier and Douady \cite{DD}, and begin by recalling the relevant material from that paper.

\subsection{Continuous fields of Hilbert modules}\label{sec:continuous-fields}

\subsubsection{Continuous fields of Banach spaces}
Let $X$ be a topological space. Let $H=(H_x)_{x \in X}$ be a family of normed $\bC$-vector spaces. Let $\Tot( H) := \coprod_{x \in X} H_x$ and let $p:\Tot (H) \to X$ be the natural map.
The vector space $\prod_{x \in X} H_x$ can be viewed as the space of (set-theoretic) cross-sections of $p$. Accordingly, we denote 
the image of $s \in \prod_{x \in X} H_x$ under the projection $\prod_{x \in X} H_x \to H_x$ by the symbol $s(x)$.
Clearly, $\prod_{x \in X} H_x$ is a module over the algebra $C(X)$ of continuous functions $X \to \bC$ (in fact, over the algebra of all $\bC$-valued functions). 

\begin{defn}\label{defn:cont-field}(Dixmier-Douady)
A \emph{continuous field of Banach spaces} or shorter \emph{Banach field} is a family $H=(H_x)_{x \in X}$ of Banach spaces, together with a subspace $\Gamma \subset  \prod_{x \in X} H_x$ of ``continuous sections'', which satisfies the following axioms:
\begin{enumerate}
\item $\Gamma$ is a $C(X,\bC)$-submodule of $\prod_{x \in X} H_x $.
\item For each $x \in X$ and $u \in H_x$, there is $s \in \Gamma$ with $s(x) =u$.
\item For each $s \in \Gamma$, the function $X \to \bR$ given by $x \mapsto \norm{s(x)}$ is continuous.
\item If $t \in \prod_{x \in X}H_x$ is such that for each $x\in X$ and each $\epsilon>0$, there is a neighborhood $U\subset X$ of $x$ and $s \in \Gamma$ such that $\sup_{y \in U}\norm{t(y)-s(y)} \leq \epsilon$ on $U$, then $t  \in \Gamma$.
\end{enumerate}
A subset $\Lambda \subset \Gamma$ is called \emph{total} if for each $x \in X$, the linear span of the set $\{ s (x) \vert s \in \Lambda\} \subset H_x$ is dense.
\end{defn}

In \cite{DD}, \S 2, a topology on $\Tot (H)$ is constructed which induces the metric topology on each fibre $H_x$, makes the map $\Tot( H) \to X$ continuous and such that $\Gamma$ is precisely the space of continuous cross-sections.
The space of sections $\Gamma$ of a Banach field has a completeness property, which follows immediately from the axioms of a Banach field:

\begin{lem}\label{convergence-of-sections}
Let $(H, \Gamma)$ be a Banach field over $X$ and let $(s_n)$ be a sequence in $\Gamma$. Suppose that for each $x \in X$ and for each $\epsilon>0$, there is an open neighborhood $x \in U\subset X$ and $n_0 $, such that for $n,m\geq n_0$, we have $\sup_{y \in U}\norm{s_n (y)-s_m (y)} \leq \epsilon$. Then $s \in \prod_{x} H_x$, defined by $s(x):= \lim_n s_n (x)$, is in $\Gamma$. 
\end{lem}

Many Banach spaces in analysis arise by completion of normed spaces, and the situation is not much different for Banach fields. Let us define an incomplete version of Banach fields. Lemma \ref{completion-of-prehilbfields} below shows how to pass to Banach fields.

\begin{defn}\label{defn:pre-banachfiled}
Let $X$ be a topological space. A \emph{pre-Banach field} is a pair $(E,\Lambda)$, consisting of a family $E=(E_x)_{x \in X}$ of normed vector spaces and a subvector space $\Lambda \subset \prod_{x \in X} E_x$ such that
\begin{enumerate}
\item For each $s \in \Lambda$, the function $X \to \bR$ given by $x \mapsto \norm{s(x)}$ is continuous.
\item For each $x \in X$, the image of the map $\Lambda \to H_x$, $s \mapsto s(x)$ is dense.
\end{enumerate}
\end{defn}

\subsubsection{Bounded operator families}

Let $(H_0,\Gamma_0) $ and $ (H_1, \Gamma_1)$ be pre-Banach fields on $X$. Any family $F:=(F_x)_{x \in X}$ of linear maps $F_x: (H_0)_x \to (H_1)_x$ induces maps $ \Tot (H_0) \to \Tot (H_1)$ and $ \prod_{x \in X} (H_0)_x \to \prod_{x \in X} (H_1)_x $, both denoted by $F$. 

\begin{defn}\label{defn-homomorphism}
Let $(E_0,\Lambda_0)$ and $(E_1, \Lambda_1)$ be pre-Banach fields over $X$. A \emph{bounded operator family} or \emph{homomorphism} $F: (E_0,\Lambda_0)\to (E_1, \Lambda_1)$ is a family $(F_x)_{x \in X}$ of bounded linear maps $F_x: (E_0)_x \to (E_1)_x$, such that
\begin{enumerate}
\item If $s \in \Lambda_0$, then $F(s) \in \Lambda_1$ and
\item the function $X \to \bR$, $x \mapsto \norm{F_x}$ is locally bounded.
\end{enumerate}
\end{defn}

It follows from \cite[Proposition 5]{DD} that in the case of Banach fields, this definition coincides with the one given in \cite[\S 4]{DD}. It is clear that linear combinations and composition of homomorphisms are again homomorphisms, and \emph{isomorphisms} are defined categorically. 
One could naively expect that if $(H_i, \Gamma_i)$, $i=0,1$, are Banach fields, the collection $(\Lin ((H_0)_x,(H_1)_x))_{x\in X}$ forms a Banach field whose continuous sections are the homomorphisms $H_0 \to H_1$, but \emph{this is not true}, unless $H_0$ is a finite dimensional vector bundle. In fact, the function $x \mapsto \norm{F_x}$ is in general not continuous, and therefore axiom (3) of Definition \ref{defn:cont-field} is violated.
This is the reason why the next Lemma does not follow immediately from Lemma \ref{convergence-of-sections}.

\begin{lem}\label{convergence-of-operators}
Let $(H_0,\Gamma_0)$ and $(H_1, \Gamma_1)$ be Banach fields over $X$ and let $F_n : H_0 \to H_1$ be a sequence of bounded operator families. Assume that
\begin{enumerate}
\item For each $x \in X$, the limit $F_x := \lim_{n \to \infty} (F_n)_x \in \Lin ((H_0)_x, (H_1)_x)$ (in the norm-topology!) exists. 
\item The convergence is locally uniform, i.e. the functions $X \to \bR$, $x \mapsto \norm{F_x-(F_n)_x}$ converge locally uniformly to $0$.
\end{enumerate}
Then $F=(F_x)_x$ is a bounded operator family.
\end{lem}

\begin{proof}
The second condition of Definition \ref{defn-homomorphism} is trivial to check. For the first condition, let $s \in \Gamma_0$. We have to check that $x \mapsto F s(x) = \lim_n F_n (s(x))$ is in $\Gamma_1$, and this follows from Lemma \ref{convergence-of-sections}.\end{proof}

\begin{lem}\cite[p. 236]{DD}\label{isomorphism-criterion}
Let $F:(H_0, \Gamma_0)\to (H_1, \Gamma_1)$ be a bounded operator family between Banach fields on $X$. 
Assume that
\begin{enumerate}
\item each $F_x$ is an invertible operator $(H_0)_x \to (H_1)_x$ and
\item the function $X \to \bR$ given by $x \mapsto \norm{F_x^{-1}}$ is locally bounded.
\end{enumerate}
Then $F$ is an isomorphism.
\end{lem}


\begin{lem}\label{completion-of-prehilbfields}
Let $X$ be a topological space and let $(E, \Lambda)$ be a pre-Banach field on $X$. 
\begin{enumerate}
\item Let $H_x:= \overline{E_x}$ be the Banach space completion of $E_x$ and $H:=(H_x)_{x\in X}$. Then there exists a unique subspace $\Gamma \subset \prod_{x \in X} H_x$, such that $\Lambda \subset \Gamma$ and such that $(H,\Gamma)$ is a Banach field. The space $\Gamma$ is the space of all $s$ such that for each $x \in X$ and $\eps>0$, there is a neighborhood $U$ and $t \in \Lambda$ such that $\sup_{y\in U} \norm{s(y)-t(y)} \leq \eps$. 
We write $\overline{(E, \Lambda)}:=(H, \Gamma)$. Moreover $\Lambda \subset \Gamma$ is a total subspace, and if $(E, \Lambda)$ is a Banach field, then $\Gamma = \Lambda$.
\item If $(E', \Gamma')$ is a Banach field and $F=(F_x)_x$ a family of bounded linear maps $E_x \to E'_x $ such that $F(\Lambda)\subset \Gamma'$ and $\norm{F_x}$ is locally bounded, then there is a unique extension of $F$ to a bounded operator family $\overline{F}:\overline{(E, \Lambda)} \to (E', \Gamma')$.
\end{enumerate}
\end{lem}

\begin{proof}
The first part is a reformulation of \cite[Proposition 3]{DD}. For the second part, let $\overline{F}_x: \overline{E_x} \to E'_x$ be the (unique) extension of the bounded operator $F_x$ to the Banach space completion. This defines $\overline{F}$, and it remains to be proven that it is a homomorphism. By construction $\norm{\overline{F}_x} = \norm{F_x}$. Moreover $\overline{F}(\Lambda) \subset \Gamma'$. A straightforward application of \cite[Proposition 5]{DD} shows that $\overline{F}(\overline{\Lambda}) \subset \Gamma'$, and this finishes the proof.
\end{proof}

\begin{corollary}\label{cor:completion-of-total-subspace}
Let $(H,\Gamma)$ be a Banach field and let $\Lambda\subset \Gamma$ be a total subspace. Then $\overline{(H, \Lambda)}=(H,\Gamma)$.
\end{corollary}

\begin{proof}
Combine Lemma \ref{completion-of-prehilbfields} and Lemma \ref{isomorphism-criterion}.
\end{proof}

\begin{corollary}\label{criterion-for-bounded-family}
Let $(H_i, \Gamma_i)$, $i=0,1$, be Banach fields on $X$. Let $F_x: (H_0)_x \to (H_1)_x$ be a collection of bounded operators such that $\norm{F_x}$ is locally bounded and such that there exists a total subspace $\Lambda \subset \Gamma_0$ with $F (\Lambda)\subset \Gamma_1$. Then $F(\Gamma_0) \subset \Gamma_1$, i.e. $F$ is a bounded operator family.
\end{corollary}

\begin{proof}
This is clear from Lemma \ref{completion-of-prehilbfields} and Corollary \ref{cor:completion-of-total-subspace}.
\end{proof}

When applied to a pre-Banach field $(H, \Lambda)$, the completion procedure does not only take the completions of the individual Banach spaces $H_x$, but also takes the completion of $\Lambda$ in a certain sense. Even if all $H_x$ are complete, $\Lambda$ might not be complete.

\begin{defn}\cite[\S 5]{DD}\label{defn:pullback-field}
Let $(H, \Gamma)$ be a Banach field over $X$ and let $f:Y \to X$ be a continuous map. The \emph{pullback} of $(H, \Gamma)$ is the following Banach field. Let $(f^* H)_y := H_{f(x)}$. For each $ s\in \Gamma$, we get a section $\tilde{s}$ of $\Tot (f^* H) \to Y$, namely $\tilde{s}(y):= s(f(y))$. Let $\tilde{\Gamma}$ be the set of those sections, and we let $f^* (H, \Gamma)$ be the completion of $(f^* H, \tilde{\Gamma})$. 
\end{defn}

The trivial Banach field $H_X$ with fibre $H$ over $X$ is the pullback of the trivial field with fibre $H$ over $*$ along the map $X \to *$. It is clear that the pullback is (strictly) functorial.

\subsubsection{Continuous fields of Hilbert-$\gA$-modules}

It is easy to define continuous fields of Banach spaces with additional structure. For example, when each of the spaces $H_x$ is a Hilbert space, we call $(H, \Gamma)$ a \emph{continuous fields of Hilbert spaces}. In that case, the function $X \to \bC; x \mapsto \scpr{s(x),t(x)}$ is continuous for all $s,t \in \Gamma$, by axiom (3) and polarization.

\begin{defn}\label{defn:continuous-field-hilbmods}
Let $\gA$ be a $\cstar$-algebra and let $X$ be a topological space. A \emph{continuous field of pre-Hilbert $\gA$-modules} is a triple $(H,\Lambda,\scpr{\_,\_})$, where $(H, \Lambda)$ is a pre-Banach field and $\scpr{\_,\_}=(\scpr{\_,\_})_x$ is a family of $\gA$-valued inner products $\scpr{\_,\_}_x$ on $H_x$, such that $\scpr{\_,\_}$ induces the norm on $E_x$ and such that for two sections $s,t \in \Lambda$, the function $X \to \gA$, $x \mapsto \scpr{s(x),t(x)}_x$ is continuous. 

If $(H,\Lambda)$ is moreover a continuous field of Banach spaces, then $(H, \Lambda, \scpr{\_,\_})$ is a \emph{continuous field of Hilbert $\gA$-modules}.
\end{defn}

It is clear that the completion procedure of Lemma \ref{completion-of-prehilbfields} produces continuous fields of Hilbert $\gA$-modules out of continuous fields of pre-Hilbert $\gA$-modules. Here is the construction of the continuous field which is of primary interest to us.

\begin{example}\label{example:basic-field-of-hilbert-modules}
Assume that $\pi: M \to X$ is a submersion, that $V \to M$ is a smooth bundle of finitely generated projective $\gA$-modules and that $g$ is a fibrewise Riemannian metric on $M$. Consider $E_x:= \Gamma_c (M_x; V_x)$ with the inner product (\ref{inner-product-for-sections-family}). Let $\Lambda \subset \prod_x E_x$ be the space $\Gamma_{cv} (M;V)$ of smooth sections with \emph{vertically compact support}, i.e. the space of all smooth sections $s: M \to V$ such that $\pi|_{\supp (s)} : \supp (s) \to X$ is a proper map. If $s \in \Lambda$, we denote by $s(x)$ the restriction $s|_{M_x}$ to the fibre. For $s,t \in \Lambda$, the function $X \to \gA$ given by $x \mapsto \scpr{s(x),t(x)}_x$ is smooth (this follows from the dominated convergence theorem). Moreover, for each $s_x \in \Gamma_c (M_x,V_x)$, there is $s\in \Lambda$ with $s|_{M_x}=s_x$ (partitions of unity). Therefore, $(E, \Lambda)$ is a continuous field of pre-Hilbert-$\gA$-modules, whose completion (in the sense of Lemma \ref{completion-of-prehilbfields}) will be 
denoted by $L^2_{\pi} (M;V)$ or $L^2_X(M;V)$. 
\end{example}


\begin{defn}\label{adjointable-ops}
Let $H_i$, $i=0,1$, be continuous fields of Hilbert $\gA$-modules on $X$ and let $F:H_0 \to H_1$ be a bounded $\gA$-linear operator family. Then $F$ is said to be \emph{adjointable} if for each $x \in X$, the operator $(F_x): (H_0)_x \to (H_1)_x$ is adjointable and if the family of adjoints $((F_x)^*)_x$ is an operator family. 
By $\Lin_{X,\gA} (H_0,H_1)$, we denote the vector space of adjointable operator families, and by $\Lin_{X,\gA} (H):= \Lin_{X,\gA} (H,H)$ the $*$-algebra of adjointable endomorphisms.
\end{defn}

If $\gA=\bC$, the condition that each $F_x$ is adjointable is of course satisfied.
If $X$ is compact, then $\Lin_{X,\gA} (H)$ is a $\cstar$-algebra, with norm $\norm{F} := \sup_{x \in X} \norm{F_x}$. 
Note that if each individual operator $F_x$ is self-adjoint, then $F$ is adjointable (and $F^*=F$), and in this case, we say that $F$ is \emph{self-adjoint}. 

\begin{lem}\label{spectral-criterion-for-invertibility}
Let $F: H \to H$ be a self-adjoint bounded operator family on a continuous field of Hilbert $\gA$-modules over a paracompact space $X$. Then $F$ is an isomorphism if and only if there is a continuous function $c: X \to (0, \infty)$ such that $\spec (F_x) \cap (-c(x),c(x)) =\emptyset$ for all $x$.
\end{lem}

\begin{proof}
If there is such a function $c$, then invertibility of $F$ follows immediately from Lemma \ref{isomorphism-criterion}. Vice versa, use the lower bounds for the spectrum coming from that Lemma and glue them together with a partition of unity.
\end{proof}

Let $(H_0,\Gamma_0)$ and $(H_1,\Gamma_1)$ be continuous fields of Hilbert $\gA$-modules. Given $s \in \Gamma_0$ and $t \in \Gamma_1$, one defines the \emph{rank one operator}
\[
\theta_{s,t} :u \mapsto \scpr{u,s} t.
\]
This is an adjointable operator family, and 
\[
\theta_{s,t}^{*}= \theta_{t,s}.
\]
Moreover, if $F$ is adjointable, then 
\[
F\theta_{s,t} = \theta_{s,Ft}; \; \theta_{s,t} F = \theta_{F^* s,t}.
\]
Therefore, the rank $1$ operator families generate a two-sided ideal in the algebra $\Lin_{X,\gA} (H)$. 

\begin{defn}\label{defn:compact-op-hilbertfield}
An operator family $F: H_0 \to H_1$ is \emph{compact} if for each $x \in X$ and $ \epsilon>0$, there exists a neighborhood $U$ of $x$ and an operator family $G$ which is a linear combination of rank one operator families such that $\sup_{y \in U}\norm{F_y-G_y}\leq  \epsilon$. The set of compact operator families is denoted $\Kom_{X,\gA} (H_0,H_1)$.
\end{defn}

\begin{remark}\label{remarks-compact-families}
\umbruch
\begin{enumerate}
\item The special case $\gA=\gC$ is due to \cite[\S 22]{DD}. 
\item In the special case $X=*$, Definition \ref{defn:compact-op-hilbertfield} yields the notion of a \emph{compact adjointable operator} on a Hilbert module \cite[Definition 15.2.6]{WO}.
\item If $F$ is compact, then so are the individual operators $F_x: (H_0)_x \to (H_1)_x$. The function $x \mapsto \norm{F_x}$ is continuous.
\item $\Kom_{X,\gA} (H)$ is a two-sided $*$-ideal in $\Lin_{X,\gA} (H)$. 
\item Limits of compact operator families (in the sense of Lemma \ref{convergence-of-operators}) are again compact.
\item There is a continuous field of $\cstar$-algebras on $X$ whose continuous sections are precisely the compact operator families, see \cite[\S 22]{DD}.
\item The condition of Definition \ref{defn:compact-op-hilbertfield} can be difficult to verify; but all we need is the following sufficient condition: if $H_i:= (V_i)_X$ are trivial fields and if $F: X \to \Kom_{\gA} (V_0, V_1)$ is a continuous map (the target being equipped with the norm topology), then $(F(x))_{x \in X}$ is a compact operator family. 
\end{enumerate}
\end{remark}

\begin{defn}\label{defn:fredholm-op-hilbertfield}
An adjointable operator family $F: H_0 \to H_1$ is \emph{Fredholm} if there is an adjointable $G: H_1 \to H_0$ such that $FG-1$ and $GF-1$ are compact (such a $G$ is called a \emph{parametrix}). 
\end{defn}

This is a strictly stronger condition that saying that $F$ is adjointable and each $F_x$ is Fredholm (where a Fredholm operator on a Hilbert module is an operator which is invertible modulo compacts). So the Fredholm condition is not a pointwise property, but on reasonable base spaces, it is a \emph{local} property.

\begin{lem}\label{lem:fedholm-local-prop}
Let $X$ be a paracompact space and let $H$ be a continuous field of Hilbert $\gA$-modules on $X$, and let $F\in \Lin_{\gA,X}(H)$. Assume that each $x \in X$ has a neighborhood $U \subset X$ such that $F|_U$ is Fredholm. Then $F$ is Fredholm.
\end{lem}

\begin{proof}
Choose local parametrices and glue them together using a partition of unity.
\end{proof}

\begin{lem}\label{lem:essential-spectral-gap}
Let $X$ be paracompact and let $H$ be a continuous field of Hilbert-$ \gA$-modules on $X$, and let $F \in \Lin_{\gA,X}(H)$ be self-adjoint. Then $F$ is Fredholm if and only if there exists an \emph{essential spectral gap} for $F$, i.e. a function $c:X \to (0, \infty)$ such that whenever $f: X \times \bR \to \gC$ is a continuous function so that $\supp(f_x)\subset (-\eps(x),\eps(x))$, then $f(F)$ is compact.
\end{lem}

Here we use functional calculus for functions $f: X \times \bR \to \gC$ to define $f(F)$, see \ref{lem:functional-caluclus-with-varying-fct} below.

\begin{proof}
By an argument completely analogous to that for \cite[Lemma 2.16]{JEInddiff}, $F$ is Fredholm if and only if there is a function $\eps: X \to (0, \infty)$ and a compact operator family $K$ such that $F^2 + K \geq \eps (x)^2$. If $h: X \times \bR \to \gC$ satisfies $\supp(h_x)\subset (-\eps(x)^2, \eps(x)^2)$, then $h(F^2+K)=0$, but $h(F^2) -h^2 (F^2+K)$ is compact. Therefore, for each even function $f$ with $\supp(f_x) \subset (-\eps(x),\eps(x))$, $f(F)$ is compact. By Stone-Weierstrass, the result follows for all such $f$.
Conversely, if $F$ has an essential spectral gap, one can built a parametrix of the form $f(F)$, using Lemma \ref{lem:functional-caluclus-with-varying-fct}.
\end{proof}

\subsubsection{The space of sections as a Hilbert module}
Assume that $X$ is a compact Hausdorff space, $\gA$ a $\cstar$-algebra and $(E,\Gamma)$ is a continuous field of Hilbert $\gA$-modules. 
Then $\Gamma$ is an $\gA(X)$-module, and it is in fact an $\gA(X)$-Hilbert module. Namely, if $s,t \in \Gamma$ are two sections, then the function $x\mapsto \scpr{s(x),t(x)}$ is a continuous function $X\to \gA$. This defines an $\gA(X)$-valued inner product on $\Gamma$, which induces a norm on $\Gamma$. Note that $\norm{s} = \sup_{x\in X} \norm{s(x)}$. The fourth axiom of a continuous field implies that $\Gamma$ is complete, in other words, a Hilbert $\gA(X)$-module. 

\begin{lem}\label{lem:operator-families-vs-adjointabkes}
Let $F:(E_0,\Gamma_0) \to (E_1,\Gamma_1)$ be an adjointable bounded operator family. Then the induced $\gA(X)$-linear map $F: \Gamma_0 \to \Gamma_1$ is adjointable. This defines a map
\[
\Lin_{\gA,X} (E_0,E_1) \to \Lin_{\gA(X)} (\Gamma_0,\Gamma_1),
\]
which is bijective, preserves adjoints and the compact operators.
\end{lem}

\begin{proof}
It is straightforward to check that $F:\Gamma_0\to \Gamma_1$ is adjointable. We construct an inverse map $\Lin_{\gA(X)} (\Gamma_0,\Gamma_1)\to \Lin_{\gA,X} (E_0,E_1) $. We can reconstruct $(E_i)_x$ from $\Gamma_i$ and the $\gA(X)$-module structure, as follows. Let $\ev_x: \gA(X)\to \gA$ be the evaluation at $x$. Then $\Gamma_0 \otimes_{\gA(X),x} \gA \to (E_0)_x$, $s \otimes a \mapsto s(x)a$ is an isomorphism. Now we map $G\in \Lin_{\gA(X)} (\Gamma_0,\Gamma_1)$ to the family of operators $E_0 \to E_1$ determined by $G_x: (E_0)_x \to (E_1)_x$, $G_x (s(x) \otimes a) := (Gs)(x) \otimes a$. By construction, $G(\Gamma_0) \subset \Gamma_1$, and the operator norms of $G_x$ are locally bounded (even globally). 

It is straightforward to prove that under this bijection, compact operators correspond to compact operator families.
\end{proof}

\begin{prop}\label{prop:hilbert-fields-vs-hilbert-modules}
Let $X$ be a compact Hausdorff space and let $\gA$ be a $\cstar$-algebra. Then each Hilbert-$\gA(X)$-module is isomorphic to one of the form $\Gamma $, where $(E,\Gamma)$ is a continuous field of Hilbert-$\gA$-modules on $X$.
\end{prop}

\begin{proof}
Let $H$ be a Hilbert-$\gA(X)$-module. For each $x \in X$, we obtain a $\gA$-valued inner product $\scpr{v,w}_x := \scpr{v,w}(x) \in \gA$, which is of course only positive semidefinite. Let $N_x \subset H$ be the kernel of $\scpr{ \_,\_ }_x$, i.e. the subspace of all $v \in H$ with $\scpr{v,v}(x)=0$. This is a closed $\gA(X)$-submodule. Then $E_x:= H/N_x$ is an $\gA(X)$-module, but it is in fact an $\gA$-module: if $f:X \to \gA$ satisfies $f(x)=0$, then $fv \in N_x$ for each $v \in H$. Therefore $E_x$ is a pre-Hilbert-$\gA$-module. Once we can show that $E_x$ is complete, the rest of the proof is easy, by the following argument. Since the inner product on $H$ is positive definite, we have $\cap_{x \in X} N_x =0$, and hence the map $H \to  \prod_{x \in X} E_x$, $v \mapsto (x \mapsto v+N_x)$ is injective. We identify $H$ with its image in $\prod_{x \in X} E_x$, and it is straightforward to see that $((E_x)_{x\in X}, H)$ is a continuous field of Hilbert-$\gA$-modules (the fourth property follows from 
the completeness of $H$).

To argue that $E_x$ is complete, let $\norm{\_}_0$ be the norm on $E_x$ induced by the $\gA$-valued inner product $\scpr{\_,\_}_x$ and let $\norm{\_}_1$ be the quotient norm on $E_x = H/N_x$ induced by the norm on $H$. We claim that $\norm{\_}_0 = \norm{\_}_1$. Since $N_x$ is closed, $(E_x, \norm{\_}_1)$ is complete, and thus $(E_x, \norm{\_}_0)$ is complete as well. To prove the equality of the norms, let $v \in H$. Then
\[
 \norm{v+N_x}_1^2 := \inf_{w \in N_x} \norm{v+w}^2 = \inf_{w \in N_x} \norm{\scpr{v+w,v+w}}_{\gA (X)} = \inf_{w \in N_x} \sup_{y \in X} \norm{\scpr{v+w,v+w}(y)}_{\gA}
\]
and
\[
\norm{v+N_x}_0^2 := \norm{\scpr{v,v}(x)}_{\gA} = \inf_{w \in N_x} \norm{\scpr{v+w,v+w}(x)}_{\gA} \leq \norm{v+N_x}_1^2
\]
(since $N_x$ is the kernel of $\scpr{\_,\_}_x$, we get $\scpr{v+w,v+w}=\scpr{v,v}$). To prove the reverse inequality, let $x \in X$ and $\epsilon>0$. There is a neighborhood $U$ of $x$ such that $\norm{\scpr{v,v}(y)}_{\gA} \leq \norm{\scpr{v,v}(x)}_{\gA} + \epsilon$ for all $y \in U$. Let $f: X \to [0,1]$ be supported in $U$ and $f(x)=1$ (Urysohn's Lemma). Then $(f-1)v \in N_x$. Therefore 
\[
\inf_{w \in N_x} \sup_{y \in X} \norm{\scpr{v+w,v+w}(y)}_{\gA} \leq \sup_{y \in X} \norm{\scpr{v+(f-1)v,v+(f-1)v}(y)}_{\gA} = \sup_{y \in X} \norm{\scpr{fv,fv}(y)}_{\gA} = 
\]
\[
= \sup_{y \in X} f(y)^2 \norm{\scpr{v,v}(y)}_{\gA} \leq   \norm{\scpr{v,v}(x)}_{\gA}+\epsilon = \norm{v+N_x}_0^2  + \epsilon.
\]
Since $\epsilon$ was arbitrary, the proof is complete.
\end{proof}

We will use this translation to compare our definition of $K$-theory with the usual one, through Kasparov theory. In our opinion, such a translation, while perfectly suitable from the viewpoint of the theory of operator algebras, makes the treatment of families of differential operators less transparent (and it only works nicely over compact base spaces). It is simpler to consider $X$ purely as a parameter space.

\subsection{Families of unbounded operators on continuous fields of Hilbert-\texorpdfstring{$\gA$}--modules}\label{subsec:families-unbounded.hilbfields}

\subsubsection{Unbounded operator families}

As the reader might already suspect, Definition \ref{defn-homomorphism} is not general enough to support the investigation of families of differential operators. We have to deal with families of unbounded operators as well. 

\begin{defn}\label{defn:unbounded-operatorfamily}
Let $(H, \Gamma)$ be a continuous field of Hilbert $\gA$-modules on the topological space $X$ and let $\Delta \subset \Gamma$ be a total subset which is at the same time a $\gA$-submodule. Assume that $W=(W_x)_{x\in X}$ is a family of $\gA$-submodules $W_x \subset H_x$ so that $\Delta \subset \prod_x W_x$. Because $\Delta $ is total, $W_x \subset H_x$ is a dense submodule. 
A family $(D_x)$ of $\gA$-linear maps $D_x: W_x \to H_x$ is called \emph{densely defined operator family with domain} $(W, \Delta)$ if $Ds \in \Gamma$ holds for all $s \in \Delta$.
We say that $D$ is \emph{symmetric} if for all sections $s ,t \in \Delta$, the identity $\scpr{Ds,t}=\scpr{s,Dt}$ holds.
\end{defn}

Note that $D_x$ is an unbounded operator with domain $W_x$, and symmetry of $D$ amounts to saying that each $D_x$ is symmetric. For each $s,t \in \Delta$, the function $x \mapsto \scpr{s(x),t(x)}_{D_x}:= \scpr{u,v}+ \scpr{D_x u,D_x v}$ is continuous, since by assumption $s,t,Ds, Dt \in \Gamma$. With the graph scalar product, the family $(W,\Delta)$ is a continuous field of pre-Hilbert $\gA$-modules on its own right, called the \emph{graph} $\gra(D)$ of $D$. 

\begin{defn}\label{defn:closed-family-unboundedop}
A densely defined symmetric operator family $D$ on $(H,\Gamma)$ with domain $(W,\Delta)$ is \emph{closed} if $(W,\Delta)$ with the graph scalar product is complete.
\end{defn}

Occasionally, we will denote $\dom (D):=(W, \Delta)$.
If $D:(W, \Delta)\to (H, \Gamma)$ is a densely defined symmetric operator family, we might form the completion of $(W, \Delta)$ with respect to the graph norm, in the sense of Lemma \ref{completion-of-prehilbfields}.
Since obviously $\norm{D_x u} \leq \norm{u}_{D_x}$, $D$ induces a (bounded!) operator family $\overline{D}: \overline{(W,\Delta)} \to (H, \Gamma)$, also by Lemma \ref{completion-of-prehilbfields}. We call $\overline{D}$ the \emph{closure} of the initial $D$. Note that $\overline{D}_x$ is the closure of the individual operator $D_x$, by construction.
Also, the inclusion $(W, \Delta) \to (H, \Gamma)$ is continuous when $(W,\Delta)$ has the graph norm, so the inclusion extends to a bounded operator family $\iota: \overline{(W, \Delta)} \to (H, \Gamma)$. Note that $\iota (\overline{\Delta})\subset \Gamma$ is a total subspace.
As in the classical theory, one needs to prove that the closure can be interpreted as a densely defined operator family.

\begin{lem}\label{lem:closure-offamily}
Let $D: (W, \Delta) \to (H,\Gamma)$ be a symmetric and densely defined operator family. Then the inclusion map $\iota:\overline{(W, \Delta)}\to (H,\Gamma)$ is injective in the sense that each $\iota_x: \overline{W_x} \to H_x$ is injective. 
\end{lem}

In particular, we can consider the closure $\overline{D}$ as a densely defined family of symmetric operators.

\begin{proof}
Symmetric operators on individual Hilbert modules are closeable with symmetric closure \cite[Lemma 2.1]{KL}, so it follows that $\overline{W_x} \to H_x$ is injective, for each $x$. This trivially implies that $\Delta \to \Gamma$ is injective as well.
\end{proof}

The following simple example shows that Definition \ref{defn:unbounded-operatorfamily} provides a nontrivial extension of the notion of a homomorphism even if all $D_x$ are bounded operators. 

\begin{example}\label{example:unbounded-family-of-bounded-ops}
Let $X=[0,1]$ and let $H_x =\bC$ for $x>0$ and $H_0=0$. Let $\Gamma :=\{s: [0,1] \to \bC \vert s(0)=0\}$. Then $(H, \Gamma)$ is a continuous field of Hilbert spaces. Let $D_x := \frac{1}{x}$, for $x>0$. Let $\dom (D):= \{ s \in \Gamma \vert \lim_{x \to 0} \frac{s(x)}{x} =0\}$. Then $(D_x)$ is a closed symmetric unbounded operator family, even if all $D_x$ are bounded (and of course $\dom (D_x)= H_x$ holds for all $x$). 
\end{example}

\subsubsection{Self-adjoint operator families}

We do not try to define the adjoint of an unbounded operator family. Instead, we directly define the notion of self-adjointness. 

\begin{defn}\label{defn-selfadjoint}
Let $(H, \Gamma)$ be a continuous field of Hilbert $\gA$-modules over the topological space $X$. A densely defined closed symmetric operator family $D: (W, \Delta) \to (H, \Gamma)$ is called \emph{self-adjoint} if each of the operators $D_x: W_x \to H_x$ is self-adjoint and regular. 
\end{defn}

The rationale behind Definition \ref{defn-selfadjoint} is provided by the Theorem \ref{thm:kaadlesch}.
Note that one could have called this property \emph{self-adjoint and regular}. However, to ease language we dropped the word ``regular''. There is little risk of confusion, because our only source of self-adjoint operators on Hilbert modules is Theorem \ref{chernoff-theorem}, which provides operators that are self-adjoint and regular at the same time. Now we discuss the example of a self-adjoint operator family which is most relevant to us.

\begin{defn}\label{defn:coercive-function-family}
Let $\pi: M \to X$ be a submersion. A function $f:M \to N$ to another manifold is called \emph{fibrewise proper} if $(\pi,f): M \to X \times N$ is proper. A fibrewise proper function $f: M \to \bR$ which is bounded from below is called a \emph{coercive function}.
\end{defn}

\begin{impexample}\label{self-adjointness-dirac-family}
Let $\pi: M \to X$ be a submersion and let $g$ be a fibrewise Riemannian metric on $M$. Let $V \to M$ be a smooth bundle of finitely generated projective $\gA$-modules and consider the continuous field $L^2_{\pi} (M;V)$ constructed in (\ref{example:basic-field-of-hilbert-modules}). Let $D: \Gamma_{cv} (M;V) \to \Gamma_{cv} (M;V)$ be a formally self-adjoint differential operator of order $1$. Then $D$ induces a family of symmetric densely defined unbounded operators on $L^2_{\pi} (M;V)$, with initial domain $\dom (D)= \Gamma_{cv} (M;V)$ and $\dom (D_x):= \Gamma_{c} (M_x;V_x)$. 
We assume that there is a smooth coercive function $f: M \to \bR$ such that the commutator $[D,f]$ is \emph{locally bounded in $X$}. Let us explain what we mean by this expression. Note that $[D,f]$ is a bundle endomorphism of $V$. It is locally bounded in $X$ if there is a continuous function $C: X \to (0,\infty)$ such that for each $p \in M$, we have $\norm{[D,f] (p)} \leq C (\pi(p))$. 

Under these assumptions, the restriction $f_x:=f|_{M_x}$ is a coercive function on $M_x$, and $[D_x,f_x]$ is bounded. Therefore, by Theorem \ref{chernoff-theorem}, the closure of $D_x$ is self-adjoint and regular. So the closure of the operator family $(D_x)_x$ is an example of a self-adjoint unbounded operator family on the continuous field $L^2_X (M;V)$ of Hilbert-$\gA$-modules.
\end{impexample}

The reason why Definition \ref{defn-selfadjoint} gives a working theory is the following global version of Proposition \ref{spectral-criterion}.

\begin{prop}\label{spectral-criterion-global}
Let $D: (W, \Delta) \to (H,\Gamma)$ be a closed symmetric densely defined operator family. The following are equivalent:
\begin{enumerate}
\item $D$ is self-adjoint in the sense of Definition \ref{defn-selfadjoint}.
\item There is $\lambda \in \bC \setminus \bR$ such that the operator families $(D\pm \lambda): (W,\Delta) \to (H,\Gamma)$ are invertible.
\item For all $\lambda \in \bC \setminus \bR$, the operator families $(D\pm \lambda): (W,\Delta) \to (H,\Gamma)$ are invertible.
\end{enumerate}
\end{prop}
\begin{proof}
\impl{3}{2} trivial. \impl{2}{1} for each $x\in X$, $D_x \pm \lambda$ is invertible, and Proposition \ref{spectral-criterion} proves that $D_x$ is self-adjoint, for all $x$. \impl{1}{3} If each $D_x$ is self-adjoint, then $D_x+\lambda: W_x \to H_x$ is invertible by Proposition \ref{spectral-criterion}. We use \ref{isomorphism-criterion} to prove that $D+\lambda$ is an isomorphism $(W, \Delta)\to (H,\Gamma)$ of Banach fields. We have to show that the operator norms of $D_x +\lambda$ and $(D_x+\lambda)^{-1}$ are uniformly bounded in $x$ (here $W_x$ has the graph norm, as usual). It is clear that
\[
\norm{(D_x+\lambda)u}  \leq (1+|\lambda|) \norm{u}_{D_x}
\]
and the operator norm of $(D_x+\lambda)^{-1}$ can be estimated by
\[
\norm{(D_x+\lambda)^{-1} u}_{D_x} \leq \norm{(D_x+\lambda)^{-1} Du} + \norm{(D_x+\lambda)^{-1} u} \leq 
\]
\[
\sup_{t \in \bR} |\frac{t}{t+\lambda}| \norm{u} + \frac{1}{|\Im (\lambda)|}\norm{u},
\]
using (\ref{coercivity}).
\end{proof}

\subsection{Spectral theory and functional calculus for operator families}\label{subsec:spectraltheoryglobal}

We now extend the functional calculus of Theorem \ref{thm:kucerovsky} to the setting of families of self-adjoint operators. 

\begin{thm}[Functional calculus for operator families]\label{thm:global-functional-calculus}
Let $(H, \Gamma)$ be a continuous field of Hilbert $\gA$-modules on $X$ and let $D$ be a family of self-adjoint operators on $H$. For $f \in \gC(\overline{\bR})$, the collection $(f(D_x))_x$ of bounded operators $H_x \to H_x$ (defined by applying the functional calculus from Theorem \ref{thm:kucerovsky} pointwise) is a bounded operator family. 
\end{thm}

\begin{proof}
It is clear that $\norm{f(D_x)} \leq \norm{f}$, so that the second condition of Definition \ref{defn-homomorphism} is satisfied. To prove that $f(D) (\Gamma) \subset \Gamma$ , let $\gI \subset \gC(\overline{\bR})$ be the subspace of functions $f$ such that $f(D)(\Gamma) \subset \Gamma$. This is a subalgebra of $\gC(\overline{\bR})$. If $f_n \in \gI$ is a sequence which converges uniformly to $f \in \gC (\overline{\bR})$, then $f\in \gI$. This is a straightforward consequence of \ref{convergence-of-operators} and the estimate $\norm{g(D_x)} \leq \norm{g}$ of Theorem \ref{thm:kucerovsky}. By Proposition \ref{spectral-criterion-global}, the functions $f_{\pm}(t)= \frac{1}{t\pm i}$ are in $\gI$, and so by the Stone-Weierstrass theorem $\gC_0 (\bR) \subset \gI$. 

Next, let $0\leq h_n\leq 1$ be a sequence of compactly supported functions which converges locally uniformly to $1$. Let $\Lambda: = \{ h_n (D) s \vert n \in \bN, s \in \Gamma\} \subset \Gamma$. By Theorem \ref{thm:kucerovsky} (3), $\Lambda$ is a total subset. If $f \in \gC(\overline{\bR})$, then $fh_n \in \gC_0(\bR)$. Therefore $f(D) h_n (D) s \in \Gamma$ for each for each $s \in \Gamma$. Therefore $f(D)(\Lambda) \subset \Gamma$, and by Corollary \ref{criterion-for-bounded-family}, it follows that $f(D) (\Gamma)\subset \Gamma$, as claimed. 
\end{proof}

It is sometimes useful to have a slightly more flexible version of the functional calculus, where the function $f$ is allowed to depend on $x \in X$ and not only on $t\in \overline{\bR}$. 

\begin{lem}\label{lem:functional-caluclus-with-varying-fct}
Let $(H,\Gamma)$ be a continuous field of Hilbert $\gA$-modules and let $D: (W, \Delta) \to (H,\Gamma)$ be a self-adjoint operator family. Let $f: X \times \bR \to \bC$ be a continuous function. For $x \in X$, set $f_x (t) := f(x,t)$. Assume that $f_x \in \gC (\overline{\bR})$ for all $x\in X$ and that $\norm{f_x}$ is locally bounded.
Then the family $f(D):=(f_x (D_x))_x$ is a bounded operator family $(H, \Gamma) \to (H,\Gamma)$. 
\end{lem}

\begin{proof}
Note that we do not assume that $f$ extends to a continuous function on $X\times \overline{\bR}$. 
It is clear that $\norm{f_x (D_x)}\leq \norm{f_x}$ is locally bounded. It it remains to be proven that $f(D) (\Gamma) \subset \Gamma$. By Corollary \ref{criterion-for-bounded-family}, it is enough to prove that $f(D)(\Delta) \subset \Gamma$. But any $s \in \Delta$ can be written in the form $(D+i)^{-1} u$ for $u \in \Gamma$, by Proposition \ref{spectral-criterion-global}.
This argument proves that it is enough to consider the function $g: X \times \bR \to \bC$, $g(x,t)= f(x,t) \frac{1}{t-i}$. The advantage is that $g$ does extend to a continuous function on $ X \times \overline{\bR}$. 
For each $x \in X$ and each $\eps>0$, there is a neighborhood $U\subset X$ of $x$ such that for all $y \in U$ and $t \in \overline{\bR}$, we have $|g_x(t)-g_y (t)| \leq \eps$. If $s \in \Gamma$, then for all $y \in U$, we have
\[
\norm{g_y (D)s(y)- g_x (D)s(y)}\leq \eps \norm{s(y)}. 
\]
But $y \mapsto g_x(D) s(y)$ is a section of $\Gamma$, and so $y \mapsto g_y(D) s(y)$ is a section as well, by Definition \ref{defn:cont-field}.
\end{proof}

\subsection{Fredholm theory for Dirac families}\label{subsec:fredholmtheory-dirac}

We are now ready to prove the main analytical results of this paper. The following assumptions are in place for the rest of this section.
\begin{enumerate}
\item $\pi: M \to X$ is a submersion with a Riemannian metric on its fibres and $V\to M$ is a smooth bundle of finitely generated projective Hilbert $\gA$-modules. We get a continuous field of Hilbert $\gA$-modules $L^2_{\pi}(M;V)$ over $X$, as in Example \ref{example:basic-field-of-hilbert-modules}.
\item $D$ is a formally self-adjoint differential operator of order $1$ on $V$. It defines an unbounded symmetric operator family with inital domain $\Gamma_{cv} (M;V)$. The closure is an unbounded closed operator family $D: \dom (D) \to L^2_{\pi}(M;V)$.
From now on, we assume that in addition $D$ is \emph{elliptic}, i.e. for each $\xi \in (T_v^* M)_x$, $x \in M$, $\xi \neq 0$, the map $\smb_{D}(\xi) (\xi): V_x \to V_x$ is an isomorphism. 
\item There is a coercive function $h: M \to \bR$ such that the commutator $[D,h]$ is locally bounded in $X$. Then the closure of $D$ is a self-adjoint operator family, by Theorem \ref{chernoff-theorem}, as explaind in (\ref{self-adjointness-dirac-family}).
\end{enumerate}

\begin{defn}\label{defn:fredholm-op-hilbertfield-unbounded}
A self-adjoint unbounded operator family $D: \dom (D)\to H$ is \emph{Fredholm} if the bounded transform $\normalize{D}$ is a (bounded) Fredholm family.
\end{defn}

We want to find conditions which ensure that $D$ is a Fredholm family or invertible. To accomplish this goal, we need a fundamental source of compactness.

\begin{thm}[Generalized Rellich theorem]\label{rellich}
Let $g \in C^{\infty}_{cv} (M)$ be a smooth function. Then the operator family $\dom (D) \to L^2_{\pi} (M;V)$, $u \mapsto gu$, is compact.
\end{thm}

\begin{proof}
Compactness of operator families is a local (in $X$) property, so that we can assume that $g$ has compact support (and not merely vertically compact support). By the inverse function theorem, $M$ can be covered by ``box neighborhoods'' $\bR^k \times \bR^d$ (the map $\pi$ is projection to $\bR^k$ in these coordinates). Moreover, we shall assume that each box neighborhood is relatively compact in $M$. Over such box neighborhoods, the bundle $V$ is isomorphic to the trivial bundle with fibre $P$, and the volume measure of the fibrewise Riemannian metric is of the form $b(x,y) dy$, for a smooth function $b$ which is uniformly bounded and uniformly bounded away from $0$.
Then we can write $g$ as a finite sum $\sum_i g_i$, and each $g_i$ has support in the unit disc of a box neighborhood.
Furthermore, it is enough to consider functions of the form $g(x,y)=f(x) h(y)$, with $g,h$ both of compact support. 

The first key ingredient is a generalization of the classical Rellich theorem due to Mishchenko and Fomenko \cite[Lemma 3.3]{MF}. Let $P$ be a finitely generated projective $\gA$-module with an inner product. On the space $C_c^{\infty}(\bR^d; P)$, consider the $\gA$-valued inner product 
\[
\scpr{u,v} := \int_{\bR^d} \scpr{u(y), v(y) }+ \sum_{j=1}^{d}\scpr{\partial_j u(y),\partial_j  v(y) } dy \in \gA.
\]
The completion of $C_c^{\infty} (\bR^d,P)$ is by definition the Sobolev space $W^1 (\bR^d, P)$. If $h \in C_c^{\infty} (\bR^d)$, then the operator $W^1 (\bR^d; P) \to L^2 (\bR^d; P)$, $u \mapsto  hu$ is a compact adjointable operator of Hilbert $\gA$-modules, by \cite[Lemma 3.3]{MF}. 
Together with Remark \ref{remarks-compact-families} (6), this implies that the map $ W^1 (\bR^d; P)_{\bR^k} \to L^2 (\bR^d; P)_{\bR^k}$ of trivial fields defined by $(x,u) \mapsto (x, f(x) hu)$ is a compact operator family. 

The second ingredient is the G\r{a}rding inequality, which says that for each elliptic $\gA$-linear operator of order $1$ on $\bR^d$ and each compact $K \subset \bR^d$, there are constants $c,C >0$ such that $c \norm{u}_{W^1}^{2} \leq \norm{u}^2_{L^2} + \norm{Du}^2_{L^2} \leq C \norm{u}^2_{W^1}$ holds for all $u$ with support in $K$. The proof of the classical G\r{a}rding inequality given in \cite[Theorem 10.4.4]{HR} carries over to the situation of $\gA$-linear operators, without change.

Hence the graph norm induced by $D$, restricted to $D^k \times D^d$, and the $W^1$-norm are equivalent. If we pick an auxiliary smooth function $a: \bR^k \times \bR^d \to [0,1]$ with compact support and $ag=g$, then we can factor the map $g:\dom (D) \to L^2_{\pi} (M;V)$ as 
\[
\dom (D) \stackrel{a}{\to}  W^1 (\bR^d, P)_{\bR^k} \stackrel{g}{\to} L^2 (\bR^d, P)_{\bR^k},
\]
i.e. as the composition of a bounded with a compact operator family. This completes the proof.
\end{proof}

\begin{proposition}\label{compactness1}
Let $f \in C_0 (\bR)$ and $g \in C^{\infty}_{cv} (M)$. Then the operator families $g f(D)$ and $f(D) g$ are compact. 
\end{proposition}

\begin{proof}
By taking adjoints, it is enough to consider the operator $g f(D)$. Let $\gJ \subset \gC(\overline{\bR})$ be the set of functions $f$ such that $g f(D)$ is compact. By the functional calculus, $\gJ$ is a closed ideal in $\gC(\overline{\bR})$. We have to prove that $\gJ=\gC_0 (\bR)$, and since $f(t)=\frac{1}{t^2+1}$ generates the ideal $\gC_0 (\bR)$ in $\gC(\overline{\bR})$, it is enough to consider the case $f(x)= \frac{1}{x^2 +1}$. Write
\[
g (D^2 +1)^{-1} = (g (D+i)^{-1}) (D-i)^{-1}.
\]
The operator family $(D-i)^{-1}$ is bounded, and we can write $g (D+i)^{-1}$ as a composition 
\[
L^2_X(M;V) \stackrel{(D+i)^{-1}}{\to} \dom (D) \stackrel{g}{\to} L^2_X(M;V)
\]
of a bounded operator family with a compact one, by Theorem \ref{rellich}. 
\end{proof}

\begin{remark}\label{rem:to-compactnesstheorems}
Theorem \ref{rellich} and Proposition \ref{compactness1} remain valid if $g$ is a smooth, compactly supported bundle endomorphism of $V$ instead of a function, by the same proof. 
\end{remark}

\begin{proposition}\label{compactness2}
Let $f \in \gC(\overline{\bR})$ and $g \in C^{\infty}_{cv} (M)$. Then $[f(D),g]$ is compact.
\end{proposition}

\begin{proof}
We fix $g$ and let $\gI \subset  \gC(\overline{\bR})$ be the set of all $f$ such that $[f(D),g]$ is compact. This is a closed subspace, and since $[f_0(D)f_1(D),g] = f_0(D) [f_1(D),g] + [f_0 (D),g] f_1 (D)$, it is also a subalgebra. By Proposition \ref{compactness1}, $\gC_0 (\bR) \subset \gI$. Since $\gC_0(\bR) \subset \gC(\overline{\bR})$ is a codimension $2$ ideal, and since $1 \in \gI$, it only remains to prove that $\normalize{t} \in \gI$. The argument for that is very similar to \cite[Proposition 17.11.3]{Bla}.

Let $Z(t):= (D^2 + t^2+1)^{-1}$ and $P= \sqrt{Z(0)}$. Note that $\normalize{D}= DP=PD$. For $u \in \Gamma_{cv} (M;V)$, compute
\begin{equation}\label{compactnessproofeq1}
[g,PD] u =  [g,P] Du + P [g,D]u.
\end{equation}
We claim that $[g,P]D$ is compact. If that is proven, \eqref{compactnessproofeq1} implies the equation $[g,PD] = [g,P]D + P[g,D]$ of bounded operator families, because $P [g,D]$ is bounded, and because $\Gamma_{cv} (M;V)$ is a total subset in $L^2_\pi (M;V)$. Furthermore, as $P[g,D]$ is compact by Proposition \ref{compactness1} and Remark \ref{rem:to-compactnesstheorems}, it also follows that $[g,PD]$ is compact.

For the proof that $[g,P]D$ is compact, we first consider the case $X=*$. Proposition \ref{prop:positive-operator-of-order-two} gives the absolutely convergent integral representation
\[
P= \frac{2}{\pi} \int_0^{\infty} Z(t) dt. 
\]
Hence for $u \in \dom (D)$, we have
\[
 [g, P] Du=\frac{2}{\pi} \int_0^{\infty} [g,Z(t)] Du dt,
\]
and a formal computation with operators yields
\begin{equation}\label{compactnessproofeq1}
[g,Z(t)]D= Z(t) [D^2,g] Z(t) D= Z(t) [D,g] DZ(t) D+Z(t)D [D,g] Z(t) D.
\end{equation}
To justify (\ref{compactnessproofeq1}), one restricts to the respective domains, using that $Z(t)$ maps $E$ to $\dom (D^2)$, by Proposition \ref{prop:square-of-selfadjoint}.
There are estimates
\[
\norm{Z(t)} \leq \frac{1}{1+t^2}; \,\norm{Z(t)D}= \norm{DZ(t)}\leq \frac{1}{2} \frac{1}{(1+t^2)^{1/2}}; \; \norm{D^2 Z(t)} \leq 1; 
\]
the second one follows from $\sup_x \frac{x}{x^2 +1+t^2} = \frac{1}{2} \frac{1}{(1+t^2)^{1/2}}$. Thus (note that $Z(t)$ and $D$ commute)
\[
\norm{[a,Z(t)]D}\leq  \norm{Z(t)} \norm{[D,a]}+\norm{Z(t)D}^2 \norm{[D,a]} \leq \frac{2}{t^2+1} \norm{[D,a]}.
\]
Therefore we obtain the absolutely convergent integral of \emph{operators}
\begin{equation}\label{compactnessproofeq2}
[a, P] D=\frac{2}{\pi} \int_0^{\infty} [a,Z(t)] D dt = \frac{2}{\pi} \int_0^{\infty} (Z(t) [D,a] D^2 Z(t) +Z(t)D [D,a] Z(t) D) dt.
\end{equation}
To show that $[a,P]D$ is compact, it remains to prove that the integrand is compact. But $Z(t) [D,a]$ is compact by Proposition \ref{compactness1} and $ D^2 Z(t)$ is bounded, so the first summand is compact. Similarly, the second summand is compact. This finishes the proof of the proposition if $X=*$. 

The changes for the parametrized case are minimal. First, compactness is a local property. By the arguments just given, the integral (\ref{compactnessproofeq2}) converges locally uniformly (in $X$).
Again by Proposition \ref{compactness1}, $\sqrt{Z(t)}[D,a]$ is globally compact (and not merely pointwise). Since the set of compact operator families is closed by Remark \ref{remarks-compact-families} (4), the arguments given in the case $X=*$ also prove that $[a,P]D$ is a compact family.
\end{proof}

Typically, the operator $\normalize{D}$ is \emph{not} Fredholm, unless further conditions on $D$ are imposed. The following two theorems give sufficient conditions. In the unparametrized case, these are well-known. The first one proves Fredholmness under a coercivity condition.

Let $h: M \to \bR$ be a function. We say that $D^2 \geq h$ if for each section $s\in \Gamma_{cv} (M; V)$ and each $x \in X$, the inequality
\[
\scpr{D^2s,s}|_x \geq \scpr{hs,s}|_x
\]
holds in the $\cstar$-algebra $\gA$. 

\begin{thm}\label{fredholmness-coercivity}
Assume that there is a coercive function $h: M \to \bR$ such that $D^2 \geq h$. Then $D$ has compact resolvent, i.e. $(D^2 +1)^{-1}$ is compact. The operator family $F:= \normalize{D}$ is Fredholm and $F^2 -1$ is compact.
\end{thm}

\begin{proof}
If $(D^2 +1)^{-1}$ is compact, then so is
\[
F^2-1 =- (D^2+1)^{-1}
\]
and hence $F$ is Fredholm, because then $F$ is a parametrix to $F$. So only the first claim of the theorem needs a proof. We let $H = L^2_X (M;V)$, with $\Gamma$ its space of continuous sections. The domain field of the closure of $D$ is denoted $(W, \Delta)$. 
Let $f \in C^{\infty}_{cv} (M)$ be a fibrewise compactly supported function such that $h|_{\supp(f)} \geq C> 0$ and $u \in \Gamma_{cv}(M;V)$. Then
\[
 \scpr{fDf u,u} = \scpr{Dfu,fu} \geq \scpr{hfu,fu}.
\]
Together with the inequality $\scpr{h fu,fu} \geq \inf_{x \in \supp(f)} h(x) \scpr{fu,fu}$ and Cauchy-Schwarz, this shows
\[
 \norm{Dfu} \norm{fu} \geq  \inf_{x \in \supp(f)} h(x) \norm{fu}^2 \geq C \norm{fu}^2.
\]
By continuity, this inequality extends to all $u \in \Delta$. Let $\norm{f} \leq \eta$, $\norm{[D, f]} \leq \delta$. Then 
\[
\norm{fu} \leq \frac{1}{C} \norm{D fu} \leq \frac{1}{C} (\norm{[D,f] u} + \norm{fDu}) \leq \frac{1}{C} (\delta \norm{ u} + \eta \norm{Du}) \leq
\]
\[
\leq \frac{1}{C} \sqrt{\delta^2+\eta^2} (\norm{u}^2+ \norm{Du}^2)^{1/2} \stackrel{\ref{estimate-unbounded-op2}}{\leq} \frac{\sqrt{2}}{C} \sqrt{\delta^2+\eta^2} \norm{u}_D =  \frac{\sqrt{2}}{C} \sqrt{\delta^2+\eta^2} \norm{(D+i)u}.
\]
An arbitrary $v\in \Gamma$ can be written as $v= (D+i)u$, $u \in \Delta$. Thus the above estimate yields
\[
\norm{f \frac{1}{D+i} v} \leq \frac{\sqrt{2}}{C} \sqrt{\delta^2+\eta^2} \norm{v}
\]
and
\[
\norm{f \frac{1}{D^2+1}} \leq \norm{f \frac{1}{D+i}} \norm{\frac{1}{D-i}} \leq \frac{\sqrt{2}}{C} \sqrt{\delta^2+\eta^2}. 
\]
Now we choose functions $g_n \in C_{cv}^{\infty} (M)$, $n \in \bN$, such that $0\leq g_n \uparrow 1$, such that $[D,g_n] \leq 1$ and such that on the support of $g_m-g_n$ (for $m \geq n$), we have $h \geq 2n$. 

We have that $g_n \frac{1}{D^2+1} u \to \frac{1}{D^2+1}u $, for each $u$, and we wish to argue that $g_n \frac{1}{D^2+1}$ converges to $\frac{1}{D^2+1}$ in norm. This then proves that $\frac{1}{D^2+1}$ is the norm limit of the compact (by Proposition \ref{compactness1}) operator families $g_n \frac{1}{D^2+1}$ and hence itself compact.

To prove the norm convergence, it is enough to prove that $g_n \frac{1}{D^2+1}$ is a Cauchy sequence. But for $m \geq n$, we have
\[
\norm{(g_m-g_n) \frac{1}{1+D^2} } \leq \frac{\sqrt{2}}{2n}\sqrt{2}=\frac{1}{n} 
\]
by the above estimates, which shows that $g_n \frac{1}{D^2+1}$ is a Cauchy sequence.
\end{proof}

We next want to relax the condition that $D^2$ is bounded below by coercive function, and prove Fredholmness under the weaker condition that $D^2$ is uniformly bounded below by a positive constant, outside a fibrewise compact set. 
\begin{thm}\label{fredholmness:invertibility-at-infty}
Assume that there exists a subset $K \subset M$ such that $\pi: K \to M$ is proper and $c>0$ such that $D^2\geq c^2$ outside $K$. Then $\normalize{D}$ is Fredholm.
\end{thm}

The proof is preceeded by a Lemma.

\begin{lem}\label{invertible-everywhere}
Assume that there is a constant $c>0$ such that $D^2 \geq c^2$. Then $F=\normalize{D}$ is invertible.
\end{lem}

\begin{proof}
By Proposition \ref{prop:positive-operator-of-order-two}, $D^2$ is invertible and $\norm{\frac{1}{D^2+1}} \leq \frac{1}{1+c^2}$. Since $1-F^2=\frac{1}{D^2+1}$, it follows that 
\[
\norm{1-F^2} \leq \frac{1}{1+c^2} < 1,
\]
so $F^2$ is invertible. But then $1=F (F (F^2)^{-1}) =( (F^2)^{-1}F)F$ shows that $F$ is invertible.
\end{proof}

\begin{proof}[Proof of Theorem \ref{fredholmness:invertibility-at-infty}]
Pick a function $f \in C^{\infty}_{cv} (M)$ such that $D^2 + f^2 \geq c^2$ on all of $M$, and such that $\norm{[D,f]} \leq c^2 /2$. Define operators
\[
E= \twomatrix{}{D}{D}{}; \; g= \twomatrix{}{-if}{if}{}; \; E_f = E+g.
\]
We want to prove that 
\[
F_0:=\normalize{E}=\twomatrix{}{\normalize{D}}{\normalize{D}}{}
\]
is Fredholm. But 
\[
E_f^2 = \twomatrix{D^2 + f^2 + i [D,f]}{}{}{D^2 + f^2 - i [D,f]} = E^2 + q; \; q= \twomatrix{f^2 + i [D,f]}{}{}{f^2 - i [D,f]}.
\]
and so $E_f^2 \geq \frac{c^2}{2}$. By Lemma \ref{invertible-everywhere}, $F_1:= \normalize{E_f}$ is invertible.
Furthermore
\[
F_0^2-F_1^2 = (1- \frac{1}{E^2+1})-(1-\frac{1}{E^2_f +1}) = \frac{1}{E^2_f +1} - \frac{1}{E^2 +1} =  \frac{1}{E^2_f +1}(1 - (E_f^2+1)\frac{1}{E^2 +1} )
\]
and 
\[
1 - (E_f^2+1)\frac{1}{E^2 +1} = 1- (E^2+1+q)\frac{1}{E^2 +1}=- q \frac{1}{E^2+1}.
\]
By Proposition \ref{compactness1}, $q \frac{1}{E^2+1}$ is compact. Therefore
\[
F_0^2-F_1^2  = -\frac{1}{E^2_f +1} q\frac{1}{E^2 +1} ,
\]
is compact, because $ \frac{1}{E^2_f +1}$ is bounded. Let $G= F_{1}^{-1}$. Then 
\[
(G^2 F_0) F_0 \sim G^2 F_1^2 =1 \text{ and } F_0 (F_0 G^2) \sim F_1^2 G =1,
\]
and this provides a left- and a right parametrix to $F_0$.
\end{proof}

The last of our Fredholmness theorems is a generalization of analytical work by Bunke \cite{Bunke}. Let $D$ be a family of Dirac operators on $M$, with $(M,D)$ complete. Then $D^2$ is self-adjoint by Proposition \ref{prop:square-of-selfadjoint}. We assume that there is a bounded self-adjoint operator $A$ such that 
\begin{equation}\label{intervible-mod-bounded}
\scpr{(D^2 + A^2) u,u} \geq c^2\scpr{u,u}
\end{equation}
holds for all $u \in \dom (D^2)$, with some $c>0$. 
\begin{lem}\label{lem:dsq+asq-selfad}
If (\ref{intervible-mod-bounded}) holds, then $D^2 +A^2$ is self-adjoint and invertible.
\end{lem}

\begin{proof}
That $D^2+A^2$ is self-adjoint is a straightforward application of the Kato-Rellich theorem (\cite[Theorem 4.5]{KL} in the Hilbert module case), and because $D^2 + A^2 \geq c^2$ is strictly positive, it is invertible by Proposition \ref{prop:positive-operator-of-order-two}.
\end{proof}

\begin{thm}\label{thm:bunke-fredholmproperty}
Assume that there exists a real-valued function $f \in C_{cv}^{\infty} (M)$ and a constant $c>0$ such that $D^2 + f^2 \geq c^2$. Then the operator $F= D\frac{1}{(D^2+f^2)^{1/2}}$ is Fredholm. Moreover, $F-F^*$ and $F^* F -1$ are compact.
\end{thm}

\begin{proof}
We first prove that $F-F^*$ is compact and use the same method as in the proof of Proposition \ref{compactness2}.
We begin with the integral formula
\[
\frac{1}{\sqrt{D^2+f^2}} = \frac{2}{\pi} \int_0^{\infty} \frac{1}{D^2+f^2+t^2} dt
\]
from Proposition \ref{prop:positive-operator-of-order-two}. Then 
\[
F-F^* = \frac{2}{\pi} \int_0^{\infty} [D,\frac{1}{D^2+f^2+t^2}] dt  = \frac{2}{\pi} \int_0^{\infty} \frac{1}{D^2+f^2+t^2}[D^2+f^2+t^2,D]\frac{1}{D^2+f^2+t^2} dt.
\]
Once we can show that the integral is absolutely convergent and the integrand is compact, we have proven that $F-F^*$ is compact, as in the proof of \ref{compactness2}. But 
\[
\frac{1}{D^2+f^2+t^2}[D^2+f^2+t^2,D]\frac{1}{D^2+f^2+t^2} = \frac{1}{D^2+f^2+t^2}[f^2,D]\frac{1}{D^2+f^2+t^2},
\]
which is bounded in operator norm by $C \frac{1}{(t^2+c^2)^2}$. Moreover $[f^2,D]$ has compact support and so $[f,D] \frac{1}{D^2+f^2+t^2}$ is compact, by Remark \ref{rem:to-compactnesstheorems}. It follows that $F^*-F$ is compact. 
Moreover
\[
FF^* = D \frac{1}{D^2+f^2} D= [D, \frac{1}{D^2+f^2}] D + \frac{1}{D^2+f^2} D^2=\frac{1}{D^2+f^2}[f^2,D] \frac{1}{D^2+f^2}D + 1 - \frac{1}{D^2 +f^2} f^2 \sim 1
\]
proves that $F$ is Fredholm.
\end{proof}

\section{\texorpdfstring{$K$}--Theory}\label{sec:KTheory}

\subsection{Definitions}\label{subsec:defns-KTH}

In this section, we describe the model for $K$-theory with which the main results of \cite{JEIndex2} are formulated and proven. Let $\gA$ be a $\cstar$-algebra, with a grading and possibly with a Real structure. We begin by fixing conventions on Clifford modules.

\begin{defn}\label{defn:pseudo-euclideanbundle}
Let $X$ be a topological space. A \emph{pseudo-Riemannian vector bundle} over $X$ is a triple $(V,g, \sigma)$, consisting of a Riemannian vector bundle $(V,g) \to X$ (of finite fibre dimension), together with a self-adjoint involution $\sigma$ on $V$. We abbreviate $V^{\sigma}=(V,\sigma)$. A special case is the trivial bundle $X \times \bR^{p,q}=X \times \bR^{p+q}$, with $\sigma = (1_p,-1_q)$.
Let $(E, \Gamma)$ be a continuous field of graded Real Hilbert-$\gA$-modules on $X$ with grading $\iota$. A \emph{$\Cl(V^{\sigma})$-structure} on $E$ is a $C(X)$-linear map $c:\Gamma(X;V) \to \Lin_{X,\gA} (E)$ which goes to Real endomorphisms, such that
\[
c(v) \iota + \iota c(v) =0; \; c(v)^* = - c(\sigma (v)); \; c(v) c(w) + c(w) c(v)= - 2 g(v, \sigma w).
\]
If $V= X\times \bR^{p,q}$, we also say \emph{$\Cl^{p,q}$-structure}. 
\end{defn}

\begin{defn}[Canonical Clifford module]\label{dfn:canonical-clnnmod}
Let $\pi:V \to X$ be a euclidean vector bundle. Let $\bS_{V,V}:=\Lambda^* V^* \otimes \bC$ be the exterior algebra bundle, with the natural Real structure, inner product and the even/odd-grading $\iota$. It has a natural $\Cl(V \oplus V^-)$-structure, as defined in \S \ref{subsec:defns-abstractFA}. We will denote by $e_V: V \to \End (\bS_{V,V})$ the action by $V$ and by $\eps_V: V^- \to \End (\bS_{V,V})$ the action of $V^-$. 
If $V= \bR^n \to *$, we denote this Clifford module simply by $\bS_{n,n}$.
\end{defn}

\begin{defn}\label{defn:tensor-product-clmod}
Let $X$ be a space and let $V,W \to X$ be two pseudo-Riemannian vector bundles. Let $(E,\eta,c)$ be a continuous field of Hilbert-$\gA$-modules with grading and $\Cl (V)$-structure. Let $(H,\iota,e)$ be a (finite-dimensional) vector bundle with grading and $\Cl (W)$-structure. One the tensor product $E \otimes H$ (which is a continuous field of Hilbert-$\gA$-modules, we define the grading $\eta \otimes \iota$ and the $\Cl (V \oplus W)$-structure $c \otimes e$, which is given by
\[
c \otimes e (v,w):= c(v) \otimes 1 + \eta \otimes e(w).
\]
\end{defn}

Note that there is a canonical isomorphism $\bS_{V,V} \otimes \bS_{W,W} \cong \bS_{V \oplus W , V \oplus W}$. 

\begin{defn}\label{defn:k-cycle}
Let $(X,Y)$ be a space pair and $V \to X$ be a pseudo-Riemannian vector bundle. A \emph{$K^V (\gA)$-cycle} on $(X,Y)$ is a tuple $(E, \iota,c,F)$, such that:
\begin{enumerate}
\item $E$ is a continuous field of graded Real Hilbert $\gA$-modules on $X$ with grading $\iota$.
\item $c$ is a $\Cl(V)$-structure on $E$.
\item $F$ is an (unbounded) odd, self-adjoint and Real Fredholm family on $E$.
\item $F$ is $\Cl (V)$-antilinear, i.e. $F c (v)+c(v) F =0$ for all $v \in V$. 
\item The family $F|_Y$ is invertible.
\end{enumerate}
A $K^V (\gA)$-cycle is \emph{degenerate} if $F$ is invertible. If $V= X \times \bR^{p,q}$, we shall say \emph{$K^{p,q}(\gA)$-cycle}.
By $\bK^V (X,Y;\gA)$, we denote the set of all $K^V (\gA)$-cycles on $(X,Y)$ and by $\bD^V(X,Y;\gA) \subset \bK^V (X,Y;\gA)$ the set of degenerate cycles.
\end{defn}

\begin{remark}\label{rem:use-universe}
When we say that $\bK^V (X,Y;\gA)$ is a \emph{set}, we of course enter the usual set-theoretical problems. For the purpose of this paper, we suggest to work in a Grothendieck universe $\cU$ \cite[p. 185--217]{SGA} which contains $\bR$ and $\gA$. Then $\cU$ contains a model for the classifying space $B U(P)$ of the unitary group of each finitely generated projective $\gA$-module. This is enough to ensure that all constructions of geometric origin lie in $\cU$. 
\end{remark}

There are obvious notions of isomorphism and direct sum of $K^V(\gA)$-cycles. Also, it is clear that $K^V(\gA)$-cycles can be pulled back along maps $f: (X',Y')\to (X,Y)$ of space pairs, which yields a map $f^*: \bK^V(X,Y;\gA)\to\bK^{f^* V}(X',Y';\gA)$. 
A \emph{concordance} between two $K^V(\gA)$-cycles $(E_i,\iota_i,c_i,F_i)$ on $(X,Y)$ is a $K^{\pr_X^* V}\gA$-cycle $(E,\iota,c,F)$ on $(X,Y) \times [0,1]$, such that $(E,\iota,c,F)|_{X \times \{i\}} = (E_i,\iota_i,c_i,F_i)$. Note that we require \emph{equality}, not isomorphism. 

The definitions given so far make sense for general space pairs $(X,Y)$, but for proofs, it is convenient to restrict to the case where $X$ is paracompact and Hausdorff, and $Y \subset X$ is closed (``$(X,Y)$ is a paracompact pair'').

\begin{lem}\label{lem:properties-concordances}\mbox{}
Let $(X,Y)$ be a paracompact pair.
\begin{enumerate}
\item Concordance of $K^V(\gA)$-cycles is an equivalence relation on $\bK^V(X,Y;\gA)$.
\item Isomorphic $K^V(\gA)$-cycles are (canonically) concordant. 
\end{enumerate}
\end{lem}

\begin{proof}
We need a procedure to glue continuous fields of Hilbert modules. Let $X$ be a normal space, $X_0, X_1 \subset X$ be closed subspaces with $X_0 \cup X_1=X$ and $X_{01} := X_{0} \cap X_1$. Let $E_0,E_1,E_{01}$ be continuous fields of Hilbert $\gA$-modules on $X_0, X_1$ and $X_{01}$ with spaces $\Gamma_0$, $\Gamma_1$ and $\Gamma_{01}$ of continuous sections and let $\phi_{01}^i: E_{01} \to E_i|_{X_{01}}$ be isomorphisms. We define a new continuous field $(E,\Gamma)$ on $X$ as follows. Let 
\[
E_x := 
\begin{cases}
(E_0)_x & x \in X_0 \setminus X_1\\
(E_1)_x & x \in X_1 \setminus X_0\\
(E_{01})_x & x \in X_{01}.
\end{cases}
\]
A section $s \in \prod_{x \in X} E_x$ is in $\Gamma$ if the the elements $s_0 \in \prod_{x\in X_0} (E_0)_x$ and $s_1 \in \prod_{x\in X_1} (E_1)_x$ defined by 
\[
s_i (x) :=
\begin{cases}
s(x) & x \in X_i \setminus X_{01}\\
\phi_{01}^i (s(x)) & x \in X_{01}
\end{cases}
\]
belong to $\Gamma_i$. Then $(E, \Gamma)$ is a continuous field of Hilbert modules. Axioms 1,3 and 4 are clear. For the verification of axiom 2, one uses \cite[Proposition 7]{DD}.

Using this gluing procedure, one can glue concordances together, which shows (1). Let $\phi:(E_0,\iota_0,c_0,F_0)\to (E_1,\iota_1,c_1,F_1)$ be an isomorphism of $K^V(\gA)$-cycles. Consider the product cycles $(E_0,\iota_0,c_0,F_0)\times [0,\frac{1}{2}]$ and $(E_1,\iota_1,c_1,F_1) \times [\frac{1}{2},1]$ and glue them together over $X\times \{\frac{1}{2}\}$, using $\phi$. This proves (2).
\end{proof}

We define 
\begin{equation}\label{defn:k-group12}
K^V(X,Y; \gA):= \bK^V(X,Y,\gA) / \text{concordance}.
\end{equation}
This is an abelian monoid under the direct sum operation. Furthermore, it is contravariantly functorial for maps of space pairs, and we have built in homotopy invariance of this functor. The next goal is to prove that $K^V(X,Y; \gA)$ is a group, i.e. that additive inverses exist. For that aim, we introduce a construction which is useful in some other contexts as well. 

\subsubsection{Extension by zero}

Let $X$ be a paracompact Hausdorff space, let $U \subset X$ be an open subset and let $j: U \to X$ be the inclusion map. 
Let $(E, \Gamma)$ be a continuous field of Hilbert $\gA$-modules on $U$. We first show how to extend $E$ to a field $j_! E$ over $X$. Define 
\[
(j_! E)_x :=
\begin{cases} E_x & x \in U\\
0 & x \not \in U. 
\end{cases}
\]
Let $\Gamma_0 =j_! \Gamma \subset \Gamma$ be the space of all sections $s$ such that the function 
\[
X \to \bR, \; x \mapsto 
\begin{cases}
\norm{s(x)} & x \in U\\
 0 & x \not \in U
\end{cases}
\]
is continuous. It is easy to see that $(j_! E, j_! \Gamma)$ is a continuous field of Hilbert $\gA$-modules, whose restriction to $U$ is $(E, \Gamma)$ (for the second axiom, one uses Urysohn's Lemma). 
Let $D: \dom (D) \to E$ be a self-adjoint unbounded operator family on $E$. We want to define an extension $j_! D: \dom (j_! D) \to j_! E$. We define $\dom(j_! D)$ as the set of all $s \in \dom (D)$ such that $s \in \Gamma_0$ and $Ds \in \Gamma_0$. Let
\[
(j_! D)_x :=
\begin{cases}
D_x: \dom (D_x) \to H_x & x \in U,\\
0 & x \not \in U.
\end{cases}
\]
Then $j_! D : \dom (j_! D) \to j_! E$ is closed, because $D: \dom (D) \to E$ was assumed to be closed. 
Note that even if $D$ was a bounded operator family, then $j_! D$ might be unbounded. The reason is that while the function $x \mapsto \norm{D_x}$ is locally bounded on $U$, the extension of this function to all of $X$ by zero might no longer be locally bounded (see Example \ref{example:unbounded-family-of-bounded-ops}). But if $\norm{D_x}$ is globally bounded on $U$, then $j_! D$ is a bounded operator family on $j_! E$. 
It is clear how to extend gradings and Clifford structures on $E$ to gradings and Clifford structures on $j_! E$. There is a slight problem to overcome. Namely, if $D$ is a Fredholm family, then $j_! D$ does not need to be a Fredholm family in general.

\begin{lem}\label{lem:extension-byzero-fredholm-family}
Let $X$ be a paracompact Hausdorff space. Let $U \subset X$ be an open subset and with inclusion map $j: U \to X$. Let $(E, \Gamma)$ be a continuous field of Hilbert $\gA$-modules on $U$.
\begin{enumerate}
\item Let $F$ be a compact operator family on $E$. Then $j_! F$ is compact if and only if the function 
\[
X \to [0,\infty), \;x\mapsto 
\begin{cases}
\norm{F_x} & x \in U\\
 0 & x \not \in U 
\end{cases}
\]
is continuous. 
\item Let $D$ be an unbounded self-adjoint Fredholm family. Assume that there is a subset $Z \subset U$ which is closed in $X$ so that $D|_{U \setminus Z}$ is invertible. Moreover, assume that each $x\in \overline{U} \setminus U$ admits a neighborhood $V \subset X$ and $c>0$ such that $\spec (D_y) \cap (-c,c) = \emptyset$ for all $y \in V \cap U$. Then $j_! D$ is Fredholm.
\end{enumerate}
\end{lem}

\begin{proof}
(1) ``Only if'' is clear since if $j_! F$ is compact, then $x\mapsto \norm{(j_! F)_x}$ is continuous. For the ``if'' direction, it is clear that the condition of Definition \ref{defn:compact-op-hilbertfield} holds at each point $x \in U$ and at each point $x\not \in \overline{U}$. Let $x \in \overline{U} \setminus U$ and let $\epsilon>0$. Then there is a neighborhood $V \subset X$ of $x$ so that $\norm{F_y-0} \leq \epsilon$ for all $y \in V$. This proves part (1).

For part (2), we use that the Fredholm property is a local property, since $X$ is paracompact (Lemma \ref{lem:fedholm-local-prop}). Again, it is clear that $j_! D$ is Fredholm on $U$ and on $X \setminus \overline{U}$. Let $x \in \overline{U} \setminus U$. Then $\norm{(\normalize{D_y})^{-1}} \leq \frac{\sqrt{1+c^2}}{c}$ for all $y \in V$, so that $j_!\normalize{D}$ is invertible over $V$ by Lemma \ref{isomorphism-criterion}. Since $\normalize{D}$ is globally bounded, we have $j_!\normalize{D} = \normalize{j_! D}$, which finishes the proof.
\end{proof}

\begin{lem}\label{k-group-groupcomplete}
Let $(X,Y)$ be a paracompact pair and let $V \to X$ be a pseudo-euclidean vector bundle. Then
\begin{enumerate}
\item Each degenerate $K^V (\gA)$-cycle on $(X,Y)$ is concordant to the trivial cycle $(0,.,.,.)$.
\item The monoid $K^V (X,Y; \gA)$ is in a group. In fact
\[
[E, \iota,c,D] = - [E, -\iota,c,D] = -[E, -\iota,-c,-D].
\] 
\end{enumerate}
\end{lem}

\begin{proof}
(1) Take a degenerate cycle, pull it back to $(0,1] \times X$ and extend it by zero along the embedding $j: (0,1] \times X \to [0,1]\times X$. It follows from Lemma \ref{lem:extension-byzero-fredholm-family} that this extension gives a concordance to the zero cycle. 

(2) Consider the cycle $(E \oplus E, \twomatrix{\iota}{}{}{-\iota},\twomatrix{c}{}{}{-c},\twomatrix{D}{}{}{-D})$ and $Q:= \twomatrix{}{1}{1}{}$. Then $\cos(t) (Q)+ \sin (t) D$ gives a concordance to the degenerate cycle with the operator $Q$. The other equation is proven in a similar way. 
\end{proof}

\subsection{Variations of the definition and comparison with the usual definition}\label{subsec:KT-variants-comparosn}

So far, we have discussed the unbounded Fredholm model for $K$-theory. There is a version which takes only bounded Fredholm families.

\begin{defn}\label{defn:k-theory-variants}
Let $(X,Y)$ be a paracompact pair and $V \to X$ a pseudo-euclidean vector bundle. Let $\gA$ be a Real graded $\cstar$-algebra. Let $\bK^V_b (X,Y;\gA) \subset \bK^V (X,Y;\gA)$ be the set of all cycles $(E,\iota,c,D)$ such that $D$ is a bounded Fredholm family. Let $K^V_b (X,Y; \gA)$ be the set of concordance classes of elements of $\bK^V_b (X,Y; \gA)$. 

Moreover, let $\bK^V_o (X,Y;\gA) \subset \bK^V_b (X,Y;\gA)$ be the set of those cycles $(E, \iota,c,D)$ such that $D^2-1$ is compact, and let $K^V_o (X,Y; \gA)$ be the set of concordance classes of such cycles. 
\end{defn}

\begin{lem}\label{lem:k-theory-different-flavour}
Let $(X,Y)$ be a paracompact space pair. Then the two maps 
\[
\varphi: K_b^V (X,Y; \gA) \to K^V (X,Y; \gA);\; \psi: K^V (X,Y; \gA) \to K_b^V (X,Y; \gA)
\]
defined by the inclusion $\bK^V_b (X,Y;\gA) \subset \bK^V (X,Y;\gA)$ and by $(E, \iota,c,D) \mapsto (E, \iota,c,\normalize{D})$, are mutually inverse bijections.
Furthermore, the inclusion $\bK^V_o(X,Y;\gA)\to \bK^V_b(X,Y;\gA)$ induces a bijection $K^V_o(X,Y;\gA)\to K^V_b(X,Y;\gA)$.
\end{lem}

\begin{proof}
The composition $ \varphi\circ \psi$ sends an (unbounded) cycle $(H, \iota,c,D)$ to the cycle $(H,\iota,c,\normalize{D})$ (which is bounded, but considered as an unbounded cycle). Consider the continuous field $\pr_X^* H$ on $X \times [0,1]$. We claim that the family 
\[
F_{s}:= \frac{D}{(1 + (sD^2))^{1/2}}, \; s  \in [0,1],
\]
is an unbounded Fredholm family on $\pr_X^* H$ which provides the desired concordance between $(H, \iota,c,D)$ and $(H,\iota,c,\normalize{D})$. For $(s,t) \in [0,1] \times \bR$, we have
\[
|\frac{x}{(1+(sx)^2)^{1/2}} (x+i)^{-1} | = | \frac{1}{(1+(sx)^2)^{1/2}}  - \frac{i}{(1+(sx)^2)^{1/2}(x+i)}  |\leq 2.
\]
Therefore Lemma \ref{lem:functional-caluclus-with-varying-fct} applies and shows that $F_s \frac{1}{D+i}$ is a bounded operator family. Hence $F_s$ is an unbounded operator family, and its closure is self-adjoint. 

That $F_s$ is a Fredholm family follows from a spectral-theoretic argument. For $x \in X$, there exists an essential spectral gap $c>0$ in a neighborhood of $x$, by Lemma \ref{lem:essential-spectral-gap}. Because the Fredholm property is local, we may assume that $D$ has a global essential spectral gap.
This means that if $b \in \gC_0(-c,c)$, then $b(D)$ is compact. Now let $a: \bR \to \bR$ be a function with $a \equiv 1$ outside $(-c,c)$ and $a \equiv 0$ on a neighborhood of $0$ and let $g(t)=\frac{a(t)}{t}$. The operator family 
\[
G:=g(D)\sqrt{1+(s^2 D^2)}
\]
is a bounded operator family on $X \times [0,1]$ by Lemma \ref{lem:functional-caluclus-with-varying-fct}. Then $FG = GF = a(D)= 1-(1-a)(D)$, and $(1-a)(D)$ is compact. This finishes the proof that $ \varphi\circ \psi$ is the identity when $Y=\emptyset$. The relative case can be done by a similar argument.

The proof that $\psi \circ \varphi $ is the identity is analogous, but easier (since it only involves bounded operator families).
The last statement follows by a spectral deformation argument, similar to \cite[Lemma A.1]{JEInddiff}.
\end{proof} 

Next, we relate the $K$-theory groups just defined to the ``usual'' $K$-theory groups, at least when $(X,Y)$ is a compact pair and when $V=\bR^{p,q}$ is the trivial bundle. 
As the ``usual'' group, we take Kasparov's $KK$-groups \cite{Kasp}. 

\begin{prop}\label{prop:comparison-k-theory-kk}
For compact pairs $(X,Y)$, there are natural isomorphisms
\[
K^{p,q} (X,Y; \gA) \cong KKR (\Cl^{p,q}; \gC_0 (X-Y;\gA)) \cong KKR (\gC; \gC_0 (X-Y;\gA)\otimes \Cl^{q,p}).
\]
\end{prop}

\begin{proof}
We first construct a map $K_o^{p,q} (X,Y; \gA) \cong KKR (\Cl^{p,q}, \gC_0 (X-Y;\gA))$, which will then be shown to be an isomorphism. Lemma \ref{lem:k-theory-different-flavour} then finishes the proof. Let $(H,\iota,c,F) \in \bK^{p,q}_o (X,Y;\gA)$ and let $\Gamma_0$ be the space of continuous sections of $H$ vanishing on $Y$. Then $\Gamma_0$ is a Hilbert-$\gC_0 (X-Y;\gA)$-module, and it has a (graded, Real) action $c:\Cl^{p,q} \to \Lin_{\gC_0 (X-Y;\gA)} (\Gamma_0)$. The operator $F$ defines a bounded adjointable operator $F\in \Lin_{\gC_0 (X-Y;\gA)} (\Gamma_0)$ which anticommutes with $\Cl^{p,q}$. Altogether $(\Gamma_0, c, F)$ is a $\Cl^{p,q}-\gC_0(X-Y; \gA)$-Kasparov cycle, and this defines the map. Using Lemma \ref{lem:operator-families-vs-adjointabkes} and Proposition \ref{prop:hilbert-fields-vs-hilbert-modules}, one proves that this is indeed an isomorphism. The second isomorphism is proven in \cite{Kasp}.
\end{proof}

If $\gA=\gC$ with the usual Real structure, then $K^{p,q} (X,Y;\gC)= KO^{p-q}(X,Y)$. More generally, the functor, defined for $p \in \bZ$ by
\[
(X,Y) \mapsto \begin{cases}
K^{p,0} (X,Y; \gA) & p \geq 0\\
K^{0,p} (X,Y;\gA) & p \leq 0
\end{cases}
\]
is the generalized cohomology theory represented by the $K$-theory spectrum of the graded $\cstar$-algebra $\gA$.

Similarly, one can prove that
\begin{equation}
 K^V (X,Y; \gA)  \cong KKR (\gC;  \Gamma_0 (X-Y;\Cl (V^- \otimes \gA))
\end{equation}
where $\Gamma_0 (X-Y;\Cl (V^- \otimes \gA)$ is the $\cstar$-algebra of continuous sections of the bundle $\Cl(V⁻)\otimes \gA \to X$ that vanish over $Y$. In this sense, the group $K^V (X,Y;\gA)$ is a twisted version of the group $K^{p,q} (X,Y;\gA)$.

\subsection{Clifford and Bott periodicity, Thom isomorphism}\label{subsec:bott-periodicity}

\subsubsection{Linear algebraic constructions}
The groups $K^V(X,Y; \gA)$ enjoy algebraic symmetries coming from the representation theory of the Clifford algebra.

Let $W \to X$ be a pseudo-euclidean and let $V \to X$ be a euclidean vector bundle. Then there is an isomorphism
\[
\Mor_V:K^W (X,Y; \gA) \to K^{W +V+V^- } (X,Y;\gA),
\]
the \emph{Morita equivalence}, defined on the cycle level by 
\[
(E, \eta, c,D) \mapsto ( E\otimes \bS_{V,V}, \eta \otimes \iota , c \otimes e, D \otimes 1).
\]
Here $(\bS_{V,V}, \iota,e)$ is the canonical $\Cl (V \oplus V^-)$-module, and the tensor product $c \otimes e$ is to be understood in the sense that for $(w,v) \in W \oplus (V \oplus V^-)$, the Clifford action is 
\[
c(w) \otimes 1 + \eta \otimes e(v).
\]
The case $V= \bR^{1,1}$ is especially important. Here $\Mor_{1,1} (E, \eta,c,D)$ is the cycle
\[
(E \otimes E, \twomatrix{\eta}{}{}{-\eta}, \twomatrix{c(w)}{}{}{c(w)},\twomatrix{}{-\eta}{\eta}{}, \twomatrix{}{\eta}{\eta}{},  \twomatrix{D}{}{}{D})
\]
(the third to the fifth entry specify the Clifford action on $W$ and on the unit basis vectors of $\bR^{1,1}$). Furthermore, there is an isomorphism $K^{V + (0,4)}(X,Y;\gA) \to K^{V + (4,0)}(X,Y;\gA)$, using the isomorphism $\Cl^{p+4,q} \cong \Cl^{p,q+4}$ of graded algebras (see e.g. \cite{JEInddiff} for details). Next, there is a \emph{Thom homomorphism}
\[
\bK^{W \oplus V^-} (X; \gA) \stackrel{\thom}{\to} \bK^{W} ((V,V_0); \gA).
\]
Here $V_0:= V \setminus DV$ is the complement of the open unit disc bundle, and we assume that $\pi:V\to X$ is a \emph{euclidean} vector bundle\footnote{It is possible to define the Thom homomorphism for pseudoeuclidean bundles, but then the target will be the \emph{Real} $KR$-groups in the sense of \cite{AtiKR}. We do not need this level of generality, though.}. 
The Thom homomorphism is defined by 
\[
(E, \iota,c,D) \mapsto (\pi^* E , \iota , c|_{W} , D + c|_{V^-}).
\]
Usually, we will use the composition of the Thom isomorphism with the Morita equivalence:
\[
K^W (X; \gA) \stackrel{\Mor_V}{\to} K^{W \oplus V \oplus V^-} (X; \gA) \stackrel{\thom}{\to} K^{W \oplus V} (V,V_0;\gA).
\]
It sends $(E, \eta,c,D)$ to the element
\[
(\pi^* (E \otimes\bS_{V,V} ), \eta \otimes \iota, c\otimes e|_{V}, D\otimes 1+ \eta\otimes e|_{V^-}).
\]
In the special case $V=\bR$, the Thom isomorphism is the \emph{Bott map}
\[
\bott: K^{W} (X; \gA) \to K^{W+(1,0)} (X \times (\bR,\bR \setminus 0); \gA),
\]
and it is easy to verify that it maps $(E,  \eta,c,D)$ to the cycle
\[
(x,t)\mapsto ( E_x \oplus  E_x, \twomatrix{\eta}{}{}{-\eta}, \twomatrix{c}{}{}{c}, \twomatrix{}{-\eta}{\eta}{}, \twomatrix{D}{t\eta}{t\eta}{-D}).
\]

\begin{thm}\label{bott-periodicity}
For compact base spaces $X$, the Thom isomorphism is an isomorphism (and hence so is the Bott map).
\end{thm}

\begin{proof}[Proof sketch]
One first deals with the case of a trivial vector bundle $V = X \times \bR$. Using Proposition \ref{prop:comparison-k-theory-kk}, one translates into the framework of Kasparov $KK$-theory. The Bott class in Kasparov theory is the element $\beta \in KK (\Cl^{1,0},\gC_0 (\bR) )$, which is represented by the Hilbert $\gC_0 (\bR)$-module $\gC_0 (\bR, \bS_{1,1})$, with the obvious $\Cl^{1,0}$-representation and the operator $\twomatrix{}{\normalize{t}}{\normalize{t}}{}$. 
It is a routine exercise with Kasparov intersection products to identify the Bott map as the intersection product with $\beta$. Using that $\beta$ is a $KK$-equivalence \cite[Theorem 7]{Kasp}, the result follows. 

By iteration, one gets the case of higher rank trivial bundles, and using that vector bundles on compact spaces have complements, the general case follows, using the transitivity of the Thom isomorphism.
\end{proof}

\bibliographystyle{plain}
\bibliography{index}

\def\cprime{$'$}
\begin{thebibliography}{10}

\bibitem{SGA}
{\em Th\'eorie des topos et cohomologie \'etale des sch\'emas. {T}ome 1:
  {T}h\'eorie des topos}.
\newblock Lecture Notes in Mathematics, Vol. 269. Springer-Verlag, Berlin-New
  York, 1972.
\newblock S{\'e}minaire de G{\'e}om{\'e}trie Alg{\'e}brique du Bois-Marie
  1963--1964 (SGA 4), Dirig{\'e} par M. Artin, A. Grothendieck, et J. L.
  Verdier. Avec la collaboration de N. Bourbaki, P. Deligne et B. Saint-Donat.

\bibitem{AtiKR}
M.~F. Atiyah.
\newblock {$K$}-theory and reality.
\newblock {\em Quart. J. Math. Oxford Ser. (2)}, 17:367--386, 1966.

\bibitem{ASIoeoIV}
M.~F. Atiyah and I.~M. Singer.
\newblock The index of elliptic operators. {IV}.
\newblock {\em Ann. of Math. (2)}, 93:119--138, 1971.

\bibitem{Bla}
Bruce Blackadar.
\newblock {\em {$K$}-theory for operator algebras}, volume~5 of {\em
  Mathematical Sciences Research Institute Publications}.
\newblock Cambridge University Press, Cambridge, second edition, 1998.

\bibitem{Bunke}
Ulrich Bunke.
\newblock A {$K$}-theoretic relative index theorem and {C}allias-type {D}irac
  operators.
\newblock {\em Math. Ann.}, 303(2):241--279, 1995.

\bibitem{Chernoff}
Paul~R. Chernoff.
\newblock Essential self-adjointness of powers of generators of hyperbolic
  equations.
\newblock {\em J. Functional Analysis}, 12:401--414, 1973.

\bibitem{Conway}
John~B. Conway.
\newblock {\em A course in functional analysis}, volume~96 of {\em Graduate
  Texts in Mathematics}.
\newblock Springer-Verlag, New York, 1985.

\bibitem{DD}
Jacques Dixmier and Adrien Douady.
\newblock Champs continus d'espaces hilbertiens et de {$C^{\ast} $}-alg\`ebres.
\newblock {\em Bull. Soc. Math. France}, 91:227--284, 1963.

\bibitem{JEIndex2}
Johannes Ebert.
\newblock Index theory in spaces of noncompact manifolds {II}: {A} stable
  homotopy version of the {A}tiyah-{S}inger index theorem.

\bibitem{JEInddiff}
Johannes Ebert.
\newblock The two definitions of the index difference.
\newblock {\em Trans. Amer. Math. Soc.}, 369(10):7469--7507, 2017.

\bibitem{Goodearl}
K.~R. Goodearl.
\newblock {\em Notes on real and complex {$C^{\ast} $}-algebras}, volume~5 of
  {\em Shiva Mathematics Series}.
\newblock Shiva Publishing Ltd., Nantwich, 1982.

\bibitem{HPS}
Bernhard Hanke, Daniel Pape, and Thomas Schick.
\newblock Codimension two index obstructions to positive scalar curvature.
\newblock {\em Ann. Inst. Fourier (Grenoble)}, 65(6):2681--2710, 2015.

\bibitem{HR}
Nigel Higson and John Roe.
\newblock {\em Analytic {$K$}-homology}.
\newblock Oxford Mathematical Monographs. Oxford University Press, Oxford,
  2000.
\newblock Oxford Science Publications.

\bibitem{KL}
Jens Kaad and Matthias Lesch.
\newblock A local global principle for regular operators in {H}ilbert
  {$C^*$}-modules.
\newblock {\em J. Funct. Anal.}, 262(10):4540--4569, 2012.

\bibitem{Kasp}
G.~G. Kasparov.
\newblock The operator {$K$}-functor and extensions of {$C^{\ast} $}-algebras.
\newblock {\em Izv. Akad. Nauk SSSR Ser. Mat.}, 44(3):571--636, 719, 1980.

\bibitem{Kuc}
Dan Kucerovsky.
\newblock Functional calculus and representations of {$C_0(\mathbb{C})$} on a
  {H}ilbert module.
\newblock {\em Q. J. Math.}, 53(4):467--477, 2002.

\bibitem{Lance}
E.~C. Lance.
\newblock {\em Hilbert {$C^*$}-modules}, volume 210 of {\em London Mathematical
  Society Lecture Note Series}.
\newblock Cambridge University Press, Cambridge, 1995.
\newblock A toolkit for operator algebraists.

\bibitem{LM}
H.~Blaine Lawson, Jr. and Marie-Louise Michelsohn.
\newblock {\em Spin geometry}, volume~38 of {\em Princeton Mathematical
  Series}.
\newblock Princeton University Press, Princeton, NJ, 1989.

\bibitem{MF}
A.~S. Mi{\v{s}}{\v{c}}enko and A.~T. Fomenko.
\newblock The index of elliptic operators over {$C^{\ast} $}-algebras.
\newblock {\em Izv. Akad. Nauk SSSR Ser. Mat.}, 43(4):831--859, 967, 1979.

\bibitem{Ros}
Jonathan Rosenberg.
\newblock {$C^{\ast} $}-algebras, positive scalar curvature, and the {N}ovikov
  conjecture.
\newblock {\em Inst. Hautes \'Etudes Sci. Publ. Math.}, (58):197--212 (1984),
  1983.

\bibitem{Schr}
Herbert Schr{\"o}der.
\newblock {\em {$K$}-theory for real {$C^*$}-algebras and applications}, volume
  290 of {\em Pitman Research Notes in Mathematics Series}.
\newblock Longman Scientific \& Technical, Harlow; copublished in the United
  States with John Wiley \& Sons, Inc., New York, 1993.

\bibitem{SoTr}
Yu.~P. Solovyov and E.~V. Troitsky.
\newblock {\em {$C^*$}-algebras and elliptic operators in differential
  topology}, volume 192 of {\em Translations of Mathematical Monographs}.
\newblock American Mathematical Society, Providence, RI, 2001.
\newblock Translated from the 1996 Russian original by Troitsky, Translation
  edited by A. B. Sossinskii.

\bibitem{Stolz2}
Stephan Stolz.
\newblock Concordance classes of positive scalar curvature metrics.
\newblock Unpublished manuscript, available
  http://www3.nd.edu/~stolz/preprint.html.

\bibitem{WO}
N.~E. Wegge-Olsen.
\newblock {\em {$K$}-theory and {$C^*$}-algebras}.
\newblock Oxford Science Publications. The Clarendon Press, Oxford University
  Press, New York, 1993.
\newblock A friendly approach.

\bibitem{Werner}
Dirk Werner.
\newblock {\em Funktionalanalysis}.
\newblock Springer-Verlag, Berlin, extended edition, 2000.

\bibitem{Wolf}
Joseph~A. Wolf.
\newblock Essential self-adjointness for the {D}irac operator and its square.
\newblock {\em Indiana Univ. Math. J.}, 22:611--640, 1972/73.

\end{thebibliography}

\end{document}